\documentclass[10pt]{amsart}
\usepackage[utf8]{inputenc}

\usepackage{amssymb,amsthm,amsmath}
\usepackage{mathrsfs}
\usepackage{enumerate}
\usepackage{graphicx,xcolor}
\usepackage{setspace}
\usepackage[hidelinks]{hyperref}
\usepackage{comment}
\usepackage{bbm}
\usepackage{float}

\usepackage{tikz}
\usetikzlibrary{arrows.meta,calc}
\usepackage{pgfplots}
\usepgfplotslibrary{fillbetween}
\usetikzlibrary{calc}
\pgfplotsset{compat=1.18}

\usepackage[paper=a4paper, left=1.2in, right=1.2in, top=1.2in, bottom=1.2in]{geometry}
\newcommand{\R}{\mathbb{R}}

\newcommand{\e}{\varepsilon}

\newcommand{\ek}{e}

\usepackage{dsfont}

\numberwithin{equation}{section} 

\newtheorem{theorem}{Theorem}[section] 
\newtheorem{lemma}[theorem]{Lemma}
\newtheorem{corollary}[theorem]{Corollary}
\newtheorem{proposition}[theorem]{Proposition}

\theoremstyle{remark}
\newtheorem{remark}[theorem]{Remark}
\newtheorem*{claim}{Claim} 
\newtheorem{conjecture}[theorem]{Conjecture}

\theoremstyle{definition}
\newtheorem{defn}[theorem]{Definition}

\hypersetup{
    colorlinks,
    linkcolor={brown!60!black},
    citecolor={red!50!black},
    urlcolor={orange!70!black}
}

\title{Restricted Hyperplane Sections of the Cross-Polytope  and the Simplex}

\author{Silouanos Brazitikos}
\author{Christos Pandis}

\address{Department of Mathematics \& Applied Mathematics, University of Crete, Voutes Campus, 70013 Heraklion, Greece}

\email{silouanb@uoc.gr,chrpandis@gmail.com}

\begin{document}
\begin{abstract}
We give a new proof of Webb's theorem on maximal central hyperplane sections of
the regular \(n\)-simplex \(\Delta_n\), viewed in its standard embedding in
\(\mathbb R^{n+1}\). A similar method also yields sharp maximal estimates for
non-central sections of \(\Delta^n\) whose distance \(d\) from the barycenter is small, namely $d< \sqrt{\frac{1}{(n+1)(2n+1)}}.$ Moreover, we obtain sharp volume estimates for central hyperplane sections of the
cross-polytope \(B_1^n\) that pass through the barycenter of a facet; see
Figure~\ref{fig:sectionsB13}.
\end{abstract}
\maketitle

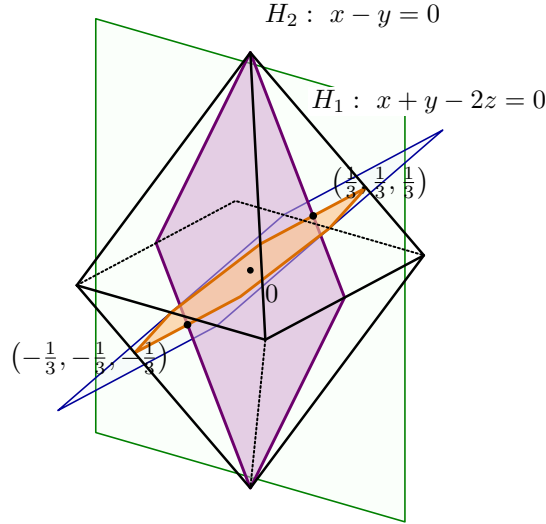
\begin{figure}[ht]
\centering
\begin{tikzpicture}[scale=0.90, line join=round, line cap=round]

\newcommand{\pt}[4]{%
  \coordinate (#1) at ({2.55*(#2) + 0.22*(#3)}, {3.2*(#4) - 1.02*(#3) + 0.22*(#2)});
}

\pt{R}{ 1}{ 0}{ 0}
\pt{L}{-1}{ 0}{ 0}
\pt{F}{ 0}{ 1}{ 0}
\pt{B}{ 0}{-1}{ 0}
\pt{T}{ 0}{ 0}{ 1}
\pt{D}{ 0}{ 0}{-1}

\pt{O}{0}{0}{0}
\pt{Gp}{ 1/3}{ 1/3}{ 1/3}
\pt{Gm}{-1/3}{-1/3}{-1/3}


\pt{U1}{ 0.82}{ 0.82}{ 0.95}
\pt{U2}{ 0.82}{ 0.82}{-0.95}
\pt{U3}{-0.82}{-0.82}{-0.95}
\pt{U4}{-0.82}{-0.82}{ 0.95}

\pt{J1}{ 0}{ 0}{ 1}
\pt{J2}{ 1/2}{ 1/2}{ 0}
\pt{J3}{ 0}{ 0}{-1}
\pt{J4}{-1/2}{-1/2}{ 0}


\pt{S1}{ 1.10}{ 0.10}{ 0.60}
\pt{S2}{ 0.10}{ 1.10}{ 0.60}
\pt{S3}{-1.10}{-0.10}{-0.60}
\pt{S4}{-0.10}{-1.10}{-0.60}

\pt{I1}{ 2/3}{ 0}{ 1/3}
\pt{I2}{ 0}{ 2/3}{ 1/3}
\pt{I3}{-1/2}{ 1/2}{ 0}
\pt{I4}{-2/3}{ 0}{-1/3}
\pt{I5}{ 0}{-2/3}{-1/3}
\pt{I6}{ 1/2}{-1/2}{ 0}

\filldraw[
  fill=green!12,
  draw=green!50!black,
  fill opacity=0.14,
  line width=0.6pt
] (U1)--(U2)--(U3)--(U4)--cycle;

\filldraw[
  fill=blue!10,
  draw=blue!60!black,
  fill opacity=0.16,
  line width=0.6pt
] (S1)--(S2)--(S3)--(S4)--cycle;

\filldraw[
  fill=violet!35,
  draw=violet!85!black,
  fill opacity=0.55,
  line width=1.1pt
] (J1)--(J2)--(J3)--(J4)--cycle;

\filldraw[
  fill=orange!45,
  draw=orange!85!black,
  fill opacity=0.62,
  line width=1.1pt
] (I1)--(I2)--(I3)--(I4)--(I5)--(I6)--cycle;

\draw[densely dotted, line width=0.75pt] (L)--(B)--(R);
\draw[densely dotted, line width=0.75pt] (F)--(D);

\draw[line width=0.95pt] (L)--(T)--(R)--(D)--(L);
\draw[line width=0.95pt] (L)--(F)--(R);
\draw[line width=0.95pt] (T)--(F);

\fill[black] (Gp) circle (1.6pt);
\fill[black] (Gm) circle (1.6pt);
\fill[black] (O)  circle (1.4pt);

\node[above right=3pt] at (Gp) {$\left(\frac13,\frac13,\frac13\right)$};
\node[below left=3pt]  at (Gm) {$\left(-\frac13,-\frac13,-\frac13\right)$};
\node[below right=2pt] at (O) {$0$};

\node[fill=white, inner sep=1pt] at ($(U1)+(0.35,0.10)$) {$H_1:\ x+y-2z=0$};
\node[fill=white, inner sep=1pt] at ($(T)+(1.5,0.55)$) {$H_2:\ x-y=0$};

\end{tikzpicture}
\caption{The maximal (purple) and minimal (orange)  sections of $B_1^3$ passing through the barycenter of a facet.}
\label{fig:sectionsB13}
\end{figure}
\tableofcontents
\section{Introduction}

The study of lower-dimensional sections and projections of convex bodies is a central theme in asymptotic geometric analysis and high-dimensional geometry. Since the seminal works of Hensley \cite{hensley1979slicing} and Ball \cite{ball1986cube} on hyperplane sections of the unit cube in $\mathbb{R}^n$, determining the extremal hyperplane sections, or more generally lower-dimensional sections and projections, of classical convex bodies such as the unit balls of $\ell_p^n$ spaces \cite{meyer1988sections,koldobsky1998application} has proved to be highly nontrivial. This problem has motivated the development and interplay of powerful analytic, geometric, and probabilistic techniques.

This line of research has led to significant activity and to numerous variants and generalizations, including, for example, complex analogues of Ball's result, due to Oleszkiewicz and Pe\l czy\'{n}ski \cite{oleszkiewicz2000polydisc}, and of Meyer--Pajor's result, due to Koldobsky and Zymonopoulou \cite{koldobsky2003extremal}; extremal perimeter sections of the cube \cite{konig2019maximal}; and non-central sections \cite{konig2021non,konig2023non,konig2024non}, among others. Given the vast literature on related topics, we refer the interested reader to the recent comprehensive survey \cite{nayar2023extremal} for a detailed historical account, the current state of the art, and numerous intriguing open problems.

Let $B_p^n$ denote the unit ball of $\ell_p^n$. In particular, $B_1^n := \operatorname{conv}\{\pm e_1,\ldots,\pm e_n\}$
denotes the cross-polytope in $\mathbb{R}^n$. Meyer and Pajor \cite{meyer1988sections} were the first to show that
\[
p \mapsto \frac{\operatorname{vol}_k(B_p^n \cap H)}{\operatorname{vol}_k(B_p^k)}
\]
is nondecreasing on $[1,\infty)$ for every $k$-dimensional subspace $H$ of $\mathbb{R}^n$; this was later extended to $[0,\infty)$ independently by Barthe \cite{barthe1995mesures} and Caetano \cite{caetano1992weyl}. More precise results are known for hyperplane sections in the range $0<p<2$. In \cite{meyer1988sections}, Meyer and Pajor showed that the minimal-volume hyperplane sections of the cross-polytope $B_1^n$ are attained in the diagonal directions, and they conjectured that the same phenomenon holds throughout the entire range $0<p<2$. This conjecture was later confirmed by Koldobsky \cite{koldobsky1998application} via a strong Schur-convexity-type result, which also shows that the maximal-volume sections are attained by hyperplanes parallel to a facet.

By contrast, the maximal-volume hyperplane sections of $B_p^n$ for $p>2$ remain unknown. Oleszkiewicz showed in \cite{oleszkiewicz2004p} that Ball's upper bound for the cube does not extend to all $p>2$, while more recently Eskenazis, Nayar and Tkocz showed in \cite{eskenazis2024resilience} the stability of the cube as a slicing maximizer when $p$ is sufficiently large. For lower bounds in the case of $2$-dimensional sections, see \cite{chasapis2022slicing}; for an upper bound in the case where $k \mid n$, see \cite{barthe1995mesures}.

Let $\Delta_n$ denote the regular $n$-dimensional simplex of side length $\sqrt{2}$, viewed in its standard embedding in $\mathbb{R}^{n+1}$, namely
\[
\Delta_n=
\left\{
x\in\mathbb{R}^{n+1}:
x_j\geq 0,\ \sum_{j=1}^{n+1}x_j=1
\right\}.
\]

By a central hyperplane section of a convex body $K$, we mean a section passing through the centroid $\mathrm{bar}(K)$. Webb \cite{webb1996central} proved that the central hyperplane sections of $\Delta_n$ having maximal volume are precisely those for which the sectioning hyperplane contains $n-1$ vertices of $\Delta_n$. Quite recently, an intriguing probabilistic proof of Webb's bound appeared in \cite{tomasz}, where the authors established sharp bounds for the negative moments of centred log-concave random variables. 

In contrast to the case of symmetric convex bodies, for which maximal-volume sections by affine subspaces of a fixed dimension always pass through the barycenter by the Brunn--Minkowski inequality, the corresponding question for the simplex is genuinely nontrivial. Webb observed that, by combining two results of Ball, one obtains that, for each fixed $1\leq k\leq n$, among all $k$-dimensional affine sections of $\Delta_n$, the maximum volume is attained precisely by its $k$-dimensional faces; see also \cite{nayar2023extremal}. 

We remark that determining the central hyperplane sections of $\Delta_n$ having minimal volume remains an open problem. It was suggested, without a formal proof, by Filliman that the minimizer is given by the vector
\[
\left(
\sqrt{\frac{n}{n+1}},
-\frac{1}{\sqrt{n(n+1)}},
\ldots,
-\frac{1}{\sqrt{n(n+1)}}
\right);
\]
see also Conjecture 4 in \cite{nayar2023extremal} and the final chapter of \cite{webb1996central} for a discussion of Filliman's results. Brzezinski~\cite{brzezinski2013volume} proved a lower bound within a factor of $1.27$ of the conjectured minimum; more precisely, he showed that for every central hyperplane $H$, the $(n-1)$-dimensional volume of $\Delta_n\cap H$ is at least
\[
\frac{e}{2\sqrt{3}}\approx \frac{1}{1.27}
\]
times that of $\Delta_n\cap H_{\mathrm{facet}}$, where $H_{\mathrm{facet}}$ is any hyperplane passing through the centroid of $\Delta_n$ and parallel to a facet. Recently, Tang \cite{tang2024simplex} improved this estimate by proving a lower bound within a factor of $1-o(1)$ of the conjectured minimum.

Moreover, K\"onig \cite{konig2021non,konig2023non} determined the maximizers in
the non-central case for \(n\geq 5\). In the range
\[
\sqrt{\frac{n-2}{3(n+1)}}<t\leq \sqrt{\frac{n}{n+1}},
\]
the maximal sections are those parallel to a face of the simplex. K\"onig also treated several low-dimensional cases. In addition, he showed that the
same vector is a local maximizer for $c(n)<t\leq d(n),$ whereas it is a local minimizer for
\[
-\frac{1}{\sqrt{n(n+1)}}<t<c(n).
\]
Here
\[
c(n)=\frac{2n+1}{n(n+2)}\sqrt{\frac{n}{n+1}},
\qquad
d(n)=\sqrt{\frac{n-1}{2(n+1)}}.
\]
\subsection{Our contribution}

The main focus of this note is the study of extremal hyperplane sections of the
regular simplex. We view the regular \(n\)-simplex \(\Delta_n\) in its standard
embedding in \(\mathbb R^{n+1}\), and we consider central hyperplane sections,
that is, sections passing through the barycenter \(\operatorname{bar}(\Delta_n)\).
Equivalently, a central section is of the form \(\Delta_n\cap a^\perp\), where
\(a\in \mathbb S^n\) satisfies  $\sum_{j=1}^{n+1}a_j=0.$
Our main result in this direction is a new proof of Webb's theorem:
for every such \(a\),
\begin{equation}\label{eq:simplex-maximal-central}
\operatorname{vol}_{n-1}(\Delta_n\cap a^\perp)
\leq
\frac{\sqrt{n+1}}{(n-1)!}\frac{1}{\sqrt{2}}.
\end{equation}
Moreover, equality holds precisely for the hyperplanes containing \(n-1\) vertices
of \(\Delta_n\), or equivalently for normals proportional, up to permutation, to $\frac{1}{\sqrt2}(1,-1,0,\ldots,0).$

The proof proceeds by reducing the problem to the analysis of the extrema of the
section-volume functional
\[
a\mapsto \operatorname{vol}_{n-1}(\Delta_n\cap a^\perp),
\]
under the constraints
\[
\sum_{j=1}^{n+1}a_j=0
\qquad\text{and}\qquad
\sum_{j=1}^{n+1}a_j^2=1.
\]
We show that every local maximum happens at a configuration with a unique negative coordinate. Once this structural
reduction is established, Webb's bound follows in a transparent way from the
explicit interpolation formula  for simplex sections.

The same approach also applies to non-central, or affine, sections of the
simplex; equivalently, to hyperplane sections at a prescribed distance from
the barycenter.  We first write affine sections in the form $H_b^d\cap \Delta_n$, where
\[
        H_b^d:=\{x\in \mathbb R^{n+1}:\langle x,b\rangle=d\},
        \qquad
        \sum_{j=1}^{n+1}b_j=0,
        \qquad
        \sum_{j=1}^{n+1}b_j^2=1.
\]
With this normalization, the parameter \(d\) is exactly the distance from the
barycenter to the hyperplane.  We treat the affine problem by passing to the  reparametrization of
these sections (see Section~\ref{sec:2}). Namely, instead of maximizing
\[
        b\longmapsto
        \operatorname{vol}_{n-1}\bigl(\Delta_n\cap H_b^d\bigr),
        \qquad
        \sum_{j=1}^{n+1}b_j=0,
        \qquad
        \sum_{j=1}^{n+1}b_j^2=1,
\]
we equivalently consider hyperplanes $H_a$ where the normal vector \(a\) satisfies
\[
        \sum_{j=1}^{n+1}a_j=\kappa,
        \qquad
        \sum_{j=1}^{n+1}a_j^2=1.
\]
Thus we solve the reparametrized maximization problem
\[
        a\longmapsto
        \operatorname{vol}_{n-1}\bigl(\Delta_n\cap H_a\bigr),
        \qquad
        \sum_{j=1}^{n+1}a_j=\kappa,
        \qquad
        \sum_{j=1}^{n+1}a_j^2=1,
\]
for \(|\kappa|<1/\sqrt{2}\).

The parameter \(\kappa\) encodes the distance of the corresponding affine
section from the barycenter. More precisely, after possibly changing the sign
of the normal vector,
\[
        |\kappa|
        =
        \frac{(n+1)d}{\sqrt{1+(n+1)d^2}},
        \qquad\text{equivalently}\qquad
        d
        =
        \frac{|\kappa|}
        {\sqrt{(n+1)\bigl(n+1-\kappa^2\bigr)}}.
\]
Consequently, the range \(|\kappa|<1/\sqrt{2}\) corresponds exactly to
\[
        0\le d<\frac{1}{\sqrt{(n+1)(2n+1)}}.
\]
In this range we obtain sharp maximal estimates for the affine problem.

We then turn to hyperplane sections of the cross-polytope \(B_1^n\) under the
additional requirement that the hyperplane pass through the barycenter of one of
its facets; see Figure~\ref{fig:sectionsB13}. The cross-polytope \(B_1^n\) has
\(2^n\) congruent simplicial facets. More precisely, each facet is an
\((n-1)\)-dimensional regular simplex of the form
\[
F_{\varepsilon}
=
\operatorname{conv}\{\varepsilon_1 e_1,\ldots,\varepsilon_n e_n\},
\]
where \(\varepsilon_i\in\{-1,1\}\) for every \(i\), and
\[
\operatorname{bar}(F_{\varepsilon})
=
\left(\frac{\varepsilon_1}{n},\ldots,\frac{\varepsilon_n}{n}\right).
\]
Let \(a\in \mathbb S^{n-1}\) be such that the hyperplane \(a^\perp\) passes through
\(\operatorname{bar}(F_{\varepsilon})\) for some choice of \(\varepsilon\). By the
symmetry of \(B_1^n\), the corresponding extremal problem is independent of the
particular facet under consideration; see~\eqref{sections of B_1^n}. Thus we can pick $\e=(1,\ldots,1)$, or equivalently $\sum_{j=1}^n a_j=0$.

For these restricted central sections of \(B_1^n\), we  show that the minimizer exhibits a parity phenomenon: in even dimensions
it is given by the half-plus/half-minus vector, whereas in odd dimensions it is
given by the almost half-plus/half-minus vector. For the maximizer, we \emph{conjecture} a change in the extremal configuration at
dimension \(n=6\). More precisely, for \(n<6\) the conjectured maximizer coincides
with the configuration corresponding to Webb's extremizer, whereas for \(n\geq6\)
it is expected to be the vector corresponding to the conjectured minimizer for
\(\Delta_n\).

Our approach to the cross-polytope problem does not proceed by directly
dealing with the section-volume map
\[
        a\mapsto \operatorname{vol}_{n-1}(B_1^n\cap a^\perp)
\]
under the constraints
\[
        \sum_{j=1}^n a_j=0,
        \qquad
        \sum_{j=1}^n a_j^2=1.
\]
Instead, we use the volume formula to reduce the problem to sharp estimates for
the elementary symmetric polynomials of the squared coordinates
\[
        e_k(a_1^2,\ldots,a_n^2),
        \qquad k=1,\ldots,n.
\]
We derive sharp upper bounds for
\[
        e_k(x_1,\ldots,x_n)
        =
        \sum_{1\leq i_1<\cdots<i_k\leq n}
        x_{i_1}x_{i_2}\cdots x_{i_k},
\]
in the special case \(x_j=a_j^2\), where the vector \(a\) satisfies
\[
        \sum_{j=1}^n a_j=0,
        \qquad
        \sum_{j=1}^n a_j^2=1.
\]
These sharp bounds are then inserted into the section-volume formula for
\(B_1^n\), yielding the desired estimates for the cross-polytope sections. This also draws a connection with the recent results of
\cite{brazitikos2025sharp} on symmetric polynomials under related constraints.
Since elementary symmetric polynomials occur naturally in many areas of
mathematics, the auxiliary estimates obtained here are of independent interest.

\section{Preliminaries}\label{sec:2}
We work in $\mathbb{R}^n$ equipped with the standard inner product
$
\langle x,y\rangle=\sum_{j=1}^n x_j y_j,
$
for $x=(x_1,\ldots,x_n)$ and $y=(y_1,\ldots,y_n)$ in $\mathbb{R}^n$. We denote by
$
\|x\|_2=\sqrt{\langle x,x\rangle}
$
the induced Euclidean norm. More generally, for $p>0$ and $x=(x_1,\ldots,x_n)\in\mathbb{R}^n$, we define
$
\|x\|_p=\left(\sum_{j=1}^n |x_j|^p\right)^{1/p}.
$
We write $B_2^n$ for the closed Euclidean unit ball in $\mathbb{R}^n$, and $\mathbb{S}^{n-1}=\partial B_2^n$ for the unit sphere. Furthermore, $e_1,\ldots,e_n$ denote the standard basis vectors of $\mathbb{R}^n$, where $e_1=(1,0,\ldots,0)$, $e_2=(0,1,0,\ldots,0)$, and so on. For a subset $A\subseteq \mathbb{R}^n$, we write
$
A^\perp=\{x\in\mathbb{R}^n:\langle x,y\rangle=0 \text{ for all } y\in A\}
$
for its orthogonal complement. In particular, for a vector $a\in\mathbb{R}^n$, we write
$
a^\perp=\{x\in\mathbb{R}^n:\langle x,a\rangle=0\}
$
for the hyperplane orthogonal to $a$. More generally, for $a\in\mathbb{R}^n$ and $t\in\mathbb{R}$, we set
$
H_a^t=\{x\in\mathbb{R}^n:\langle x,a\rangle=t\},
$
which is a hyperplane orthogonal to $a$; if $a\in\mathbb{S}^{n-1}$, then $H_a^t$ lies at signed distance $t$ from the origin. For $1\leq k\leq n$, we denote by $\operatorname{vol}_k(\cdot)$ the $k$-dimensional Lebesgue measure, identified with the $k$-dimensional Hausdorff measure and normalized so that cubes of side length $1$ have volume $1$. Finally, recall that a body in $\mathbb{R}^n$ is a compact set with nonempty interior.

\subsection{Formulae for sections}

Let $K$ be a centrally symmetric star body in $\mathbb{R}^n$. Its Minkowski functional is defined by
$
\|x\|_K:=\inf\{r>0:x\in rK\}.
$
This functional is continuous on $\mathbb{R}^n$, positively homogeneous of degree $1$, and vanishes only at $x=0$. Koldobsky proved in \cite{koldobsky1998application} that the $(n-1)$-dimensional volume of the section of $K$ by the hyperplane $a^\perp$ can be expressed in terms of the Fourier transform, in the sense of distributions, of the function $x\mapsto \|x\|_K^{-n+1}$. More precisely, for every $a\in \mathbb{S}^{n-1}$,
\begin{equation}\label{sect star bodies}
    \operatorname{vol}_{n-1}(K\cap a^{\perp})=\frac{1}{\pi(n-1)}\bigl(\|x\|_K^{-n+1}\bigr)^{\wedge}(a).
\end{equation}

The Fourier transforms of powers of the norms of the spaces $\ell_p^n$ were computed in \cite{koldobsky1992schoenberg} for $0<p<\infty$ (and for $p=\infty$ in \cite{koldobsky1994characterization}). Combining these formulas with \eqref{sect star bodies} yields
\begin{equation}\label{section formula B_p}
    \operatorname{vol}_{n-1}(B_p^n\cap a^{\perp})
    =
    \frac{p}{\pi(n-1)\Gamma\!\left(\frac{n-1}{p}\right)}
    \int_{0}^{\infty}\prod_{j=1}^n \gamma_p(t a_j)\,dt,
\end{equation}
where $\gamma_p$ denotes the Fourier transform of the function $x\mapsto e^{-|x|^p}$ on $\mathbb{R}$. Formula \eqref{section formula B_p} was first established by Meyer and Pajor \cite{meyer1988sections} for $1\leq p\leq 2$, and was later extended to all $p>0$ by Koldobsky \cite{koldobsky1998application}.

In the special case $p=1$, corresponding to the cross-polytope, \eqref{section formula B_p} reduces to
\begin{equation}\label{sections of B_1^n}
    \operatorname{vol}_{n-1}(B_1^n\cap a^{\perp})
    =
    \frac{2^n}{\pi (n-1)!}
    \int_{0}^{\infty}\prod_{j=1}^n \frac{1}{1+a_j^2 t^2}\,dt.
\end{equation}
For the regular simplex $\Delta_n$, the following formula for hyperplane sections played a central role in Webb's work. If $a\in \mathbb{S}^{n}$ satisfies $\sum_{j=1}^{n+1} a_j=0$, then
\begin{equation}\label{Webb slicing f(0)}
    \operatorname{vol}_{n-1}(\Delta_n\cap a^{\perp})
    =
    \frac{\sqrt{n+1}}{(n-1)!}\,
    f_{\sum_{j=1}^{n+1} a_j \mathcal{E}_j}(0),
\end{equation}
where $f_{\sum_{j=1}^{n+1} a_j\mathcal{E}_j}$ denotes the density of the random variable $\sum_{j=1}^{n+1} a_j\mathcal{E}_j$, and $\mathcal{E}_1,\ldots,\mathcal{E}_{n+1}$ are independent standard exponential random variables, that is, random variables with density $e^{-x}\mathbbm{1}_{(0,\infty)}(x)$ on $\mathbb{R}$.

By Fourier inversion, \eqref{Webb slicing f(0)} can be rewritten as
\begin{equation}
    \operatorname{vol}_{n-1}(\Delta_n\cap a^{\perp})
    =
    \frac{\sqrt{n+1}}{(n-1)!}\frac{1}{2\pi}
    \int_{\mathbb{R}}\prod_{j=1}^{n+1}\frac{1}{1+i a_j t}\,dt.
\end{equation}

A further expression for $f_{\sum_{j=1}^{n+1} a_j\mathcal{E}_j}(0)$, valid under the constraints $\sum_{j=1}^{n+1} a_j=0$ and $\sum_{j=1}^{n+1} a_j^2=1$, can be obtained by contour integration; see \cite{webb1996central}. Namely,
\begin{equation}\label{f(0) contour}
    f_{\sum_{j=1}^{n+1} a_j\mathcal{E}_j}(0)
    =
    \sum_{\substack{j=1\\ a_j>0}}^{n+1}
    \frac{1}{a_j}
    \prod_{\substack{k=1\\ k\neq j}}^{n+1}\frac{a_j}{a_j-a_k}
    =
    -\sum_{\substack{j=1\\ a_j<0}}^{n+1}
    \frac{1}{a_j}
    \prod_{\substack{k=1\\ k\neq j}}^{n+1}\frac{a_j}{a_j-a_k}.
\end{equation}

Later, Dirksen \cite{dirksen2016sections} derived a closed formula for arbitrary hyperplane sections, not necessarily passing through the barycenter. For $a\in \mathbb{S}^{n}$,
\begin{equation}
   \operatorname{vol}_{n-1}(\Delta_n \cap H_a)
   =
   \frac{\sqrt{n+1-\left(\sum_{j=1}^{n+1}a_j\right)^2}}{(n-1)!}\frac{1}{2\pi}
   \int_{\mathbb{R}}\prod_{j=1}^{n+1}\frac{1}{1+i a_j s}\,ds.
\end{equation}

Now consider an arbitrary section of $\Delta_n$. Such a section can be represented in two distinct ways; (see Figures~\ref{fig:HatSection} and \ref{fig:HbSection}) as an affine section $H_a^t\cap \Delta_n$ with $\|a\|=1$, $\sum_{j=1}^{n+1}a_j=0$ and $|t|=\text{dist}(\text{bar}(\Delta^n),H_a^t\cap \Delta_n)$ or a section $H_b\cap \Delta_n$ where $b\in \R^{n+1}$ and $H_a^t\cap\Delta^n=H_b\cap \Delta_n$. 

One can convert the representations into each other. Indeed, for the one direction, let $H_b\cap \Delta_n$, with $\|b\|=1$ be an arbitrary section. Then, we can find $a\in \R^{n+1}$ such that $\|a\|=1$ and $\sum_{j=1}^{n+1}a_j=0$ so that 
$H_b\cap \Delta_n = H_a^t\cap \Delta_n.$ (see~\cite{dirksen2016sections}) Moreover, we also know that
\begin{equation}\label{dist from bary}
|t|= \operatorname{dist}(\operatorname{bar}(\Delta_n),H_b\cap\Delta_n)
 =
 \frac{\left|\sum_{j=1}^{n+1}b_j\right|}{\sqrt{(n+1)\left(n+1-\left(\sum_{j=1}^{n+1}b_j\right)^2\right)}}.
\end{equation}
Thus, prescribing the distance from the centroid is equivalent to prescribing the quantity
$ \kappa:=\sum_{j=1}^{n+1} b_j.$

Again, by contour integrating, we obtain, when all $x_j>0$ are pairwise distinct and  at least one is positive (or respectively for the other case)  (see \cite{dirksen2016sections})
\begin{align}\label{affine sect contoured}
\operatorname{vol}_{n-1}\left(\Delta_n \cap H_a \right)
   &=  \frac{\sqrt{n+1-\left(\sum_{j=1}^{n+1}a_j\right)^2}}{(n-1)!}\sum_{a_j>0}\frac{1}{a_j}\prod_{k\neq j}\frac{a_j}{a_j-a_k}\\
    &=-\frac{\sqrt{n+1-\left(\sum_{j=1}^{n+1}a_j\right)^2}}{(n-1)!}\sum_{a_j<0}\frac{1}{a_j}\prod_{k\neq j}\frac{a_j}{a_j-a_k} .
\end{align}

Finally, we note that all of the above formulas are invariant under permutations of the coordinates, and therefore so are the corresponding extremizers.

\begin{figure}[ht]
\centering

\begin{minipage}{0.48\textwidth}
\centering
\begin{tikzpicture}[
    x=0.95cm,y=0.95cm,
    line join=round,line cap=round,
    axis/.style={-{Stealth[length=2.2mm]},draw=gray!35,line width=0.45pt},
    simplex/.style={draw=black,line width=0.7pt},
    ray/.style={draw=black,line width=0.7pt},
    section/.style={draw=black,line width=1.8pt},
    vec/.style={-{Stealth[length=2.8mm]},draw=black,line width=0.95pt}
]

\coordinate (O)  at (0,0);

\coordinate (X3) at (0,4.10);
\coordinate (X2) at (4.55,2.05);
\coordinate (X1) at (4.55,-2.05);

\coordinate (AX3) at ($(O)!1.14!(X3)$);
\coordinate (AX2) at ($(O)!1.16!(X2)$);
\coordinate (AX1) at ($(O)!1.16!(X1)$);

\draw[axis] (O) -- (AX1) node[right] {$x_1$};
\draw[axis] (O) -- (AX2) node[right] {$x_2$};
\draw[axis] (O) -- (AX3) node[above] {$x_3$};

\draw[simplex] (X3)--(X2)--(X1)--cycle;

\coordinate (A1) at ($(X3)!0.58!(X2)$);
\coordinate (A2) at ($(X3)!0.58!(X1)$);
\draw[ray] (O)--(A1);
\draw[ray] (O)--(A2);
\draw[section] (A1)--(A2);

\coordinate (A1) at ($(X3)!0.58!(X2)$);
\coordinate (A2) at ($(X3)!0.58!(X1)$);
\draw[ray] (O)--(A1);
\draw[ray] (O)--(A2);
\draw[section] (A1)--(A2);

\coordinate (T1) at ($(X3)!0.22!(X2)$);
\coordinate (T2) at ($(X3)!0.22!(X1)$);
\coordinate (Ot) at ($(O)!0.58!(X3)$);

\draw[ray] (Ot)--(T1);
\draw[ray] (Ot)--(T2);
\draw[section] (T1)--(T2);

\draw[vec] (O) -- (-1.75,1.45);
\node at (-0.45,0.78) {$a$};

\node at ($(A1)!0.50!(A2)+(1.05,1.38)$) {$H_a\cap S$};
\node at ($(T1)!0.50!(T2)+(0.00,0.92)$) {$H_a^t\cap S$};
\node at ($(O)!0.55!(A2)+(0.20,-0.30)$) {$H_a$};
\node[left] at ($(Ot)!0.55!(T2)+(-0.48,0.08)$) {$H_a^t$};
\end{tikzpicture}

\vspace{0.35em}
\refstepcounter{figure}\label{fig:HatSection}
Figure \thefigure.: Section $H_a^t\cap S$

\end{minipage}
\hfill
\begin{minipage}{0.48\textwidth}
\centering
\begin{tikzpicture}[
    x=0.95cm,y=0.95cm,
    line join=round,line cap=round,
    axis/.style={-{Stealth[length=2.2mm]},draw=gray!35,line width=0.45pt},
    simplex/.style={draw=black,line width=0.7pt},
    ray/.style={draw=black,line width=0.7pt},
    section/.style={draw=black,line width=1.8pt},
    vec/.style={-{Stealth[length=2.8mm]},draw=black,line width=0.95pt}
]

\coordinate (O)  at (0,0);

\coordinate (X3) at (0,4.10);
\coordinate (X2) at (4.55,2.05);
\coordinate (X1) at (4.55,-2.05);

\coordinate (AX3) at ($(O)!1.14!(X3)$);
\coordinate (AX2) at ($(O)!1.16!(X2)$);
\coordinate (AX1) at ($(O)!1.16!(X1)$);

\draw[axis] (O) -- (AX1) node[right] {$x_1$};
\draw[axis] (O) -- (AX2) node[right] {$x_2$};
\draw[axis] (O) -- (AX3) node[above] {$x_3$};

\draw[simplex] (X3)--(X2)--(X1)--cycle;

\coordinate (A1) at ($(X3)!0.60!(X2)$);
\coordinate (A2) at ($(X3)!0.60!(X1)$);
\draw[ray] (O)--(A1);
\draw[ray] (O)--(A2);
\draw[section] (A1)--(A2);

\coordinate (B1) at ($(X3)!0.24!(X2)$);
\coordinate (B2) at ($(X3)!0.24!(X1)$);
\draw[ray] (O)--(B1);
\draw[ray] (O)--(B2);
\draw[section] (B1)--(B2);

\draw[vec] (O) -- (-1.35,1.35);
\draw[vec] (O) -- (-1.95,0.45);
\node at (-0.45,0.72) {$a$};
\node at (-1.00,0.33) {$b$};

\node at ($(A1)!0.50!(A2)+(1.20,1.20)$) {$H_a\cap S$};
\node at ($(B1)!0.50!(B2)+(0.18,0.92)$) {$H_b\cap S$};
\node at ($(O)!0.58!(A2)+(0.22,-0.22)$) {$H_a$};
\node at ($(O)!0.56!(B2)+(-0.48,0.10)$) {$H_b$};

\end{tikzpicture}

\vspace{0.35em}
\refstepcounter{figure}\label{fig:HbSection}
Figure \thefigure.: Section $H_b\cap S$
\end{minipage}

\end{figure}

\subsection{Schur-convexity and majorization}

Schur-convexity-type arguments have recently appeared in probabilistic settings
(see, for example, \cite{chasapis2024haagerup,chasapis2023entropies}), leading to
sharp results ranging from moment comparison inequalities to entropy inequalities.
For a concise exposition on majorization and Schur-convexity, we refer to
Chapter~II of \cite{bhatia1997matrix}. We recall here the basic notions that
will be used throughout the paper.

\begin{defn}[Decreasing rearrangement]
Given $x = (x_1, \ldots, x_n)\in\R^n$, we denote by
$x^{\ast} = (x_1^{\ast}, \ldots, x_n^{\ast})$ its decreasing rearrangement, i.e.
\[
x_1^{\ast} \ge x_2^{\ast} \ge \cdots \ge x_n^{\ast}.
\]
\end{defn}

\begin{defn}[Majorization]
For any two vectors $x, y \in \mathbb{R}^n$, we say that $x$ is majorized by $y$,
and write $x \prec y$, if
\[
\sum_{i=1}^n x_i = \sum_{i=1}^n y_i
\quad \text{and} \quad
\sum_{i=1}^k x_i^{\ast} \le \sum_{i=1}^k y_i^{\ast}
\quad \text{for every } k = 1, 2, \ldots, n.
\]
\end{defn}

As a direct consequence, for every vector $a = (a_1, \ldots, a_n) \in \mathbb{R}^n_+$
such that $\sum_{i=1}^n a_i = 1$, we have
\begin{equation}\label{maj sum=1}
\left( \frac{1}{n}, \ldots, \frac{1}{n} \right)
\prec (a_1, \ldots, a_n)
\prec (1, 0, \ldots, 0).
\end{equation}
More specifically, if $\sum_{i=1}^n a_i^2 = 1$, then
\begin{equation}\label{maj seq}
\left( \frac{1}{n}, \ldots, \frac{1}{n} \right)
\prec (a_1^2, \ldots, a_n^2)
\prec (1, 0, \ldots, 0).
\end{equation}

\begin{defn}[Schur-convexity/concavity]
A function $f : \mathbb{R}^n \to \mathbb{R}$ is said to be Schur-convex (resp.
Schur-concave) if $x \prec y$ implies $f(x) \le f(y)$ (resp. $f(x) \ge f(y)$).
\end{defn}

A central criterion for establishing the Schur-convexity or Schur-concavity of
a function is due to Schur and Ostrowski.

\begin{theorem}[Schur--Ostrowski]\label{thm:Schur-Ostrowski}
Let $f:\mathbb{R}^n\rightarrow \mathbb{R}$ be a symmetric function with
continuous partial derivatives. Then $f$ is Schur-convex (resp. Schur-concave)
if and only if
\[
(x_i - x_j) \left( \frac{\partial f}{\partial x_i} -
\frac{\partial f}{\partial x_j} \right) \ge 0
\quad (\text{resp. } \le 0)
\]
for all $x \in \mathbb{R}^n$ and for all $1 \le i, j \le n$.
\end{theorem}
\subsection{Elementary Symmetric Polynomials}
The elementary symmetric polynomials in $n$ variables $x_1,\ldots,x_n$, denoted by $e_k(x_1,\ldots,x_n)$, for $k=1,\ldots,n$ are defined by \[
        e_k(x_1,\ldots,x_n)
        =
        \sum_{1\le i_1<\cdots<i_k\le n}
        x_{i_1}x_{i_2}\cdots x_{i_k}.
\]
From a generating function standpoint (see~\cite{macdonald1995symmetric})
\[
        \prod_{j=1}^n(1+x_jt)
        =
        \sum_{k=0}^n e_k(x_1,\ldots,x_n)t^k.
\]
It is well established that 
\begin{proposition}\label{e_k Schur}
    The function $e_k(x_1,\ldots,x_n)$  is Schur concave on $\mathbb{R}^n_+$.
\end{proposition}

\section{Simplex slicing via Lagrange multipliers} 

Our main structural result shows that, up to changing the orientation of the normal
vector, since $a,-a$ generate the same hyperplane, any maximizer has exactly one negative coordinate. For convenience, we work with \(n\) coordinates instead of \(n+1\). 

We first carry out the argument on a fixed support, ignoring zero coordinates since it just reduces to lower dimension problem.
 In the end we also add a zero group.

\begin{theorem}\label{thm:mu-positive-one-negative}
Let
\[
I_n(a):=\int_{-\infty}^{\infty}\prod_{j=1}^n \frac{1}{1+i t a_j}\,dt,
\qquad
\mathcal{M}:=\Big\{a\in\mathbb R^n:\ \sum_{j=1}^n a_j=0,\ \sum_{j=1}^n a_j^2=1\Big\}.
\]
Let \(a\in\mathcal{M}\) be a smooth local maximum of \(I_n\) on \(\mathcal{M}\), and let
\(\lambda,\mu\in\mathbb R\) be the Lagrange multipliers:
\[
\partial_{a_j}I_n(a)=2\lambda a_j+\mu,
\qquad j=1,\dots,n.
\]
Assume that \(\mu>0\).
Then \(a\) has exactly one negative coordinate.
\end{theorem}

\begin{remark}
It is obvious that \(I_n(a)\in \mathbb{R}\). Indeed, \(\overline{I_n(a)}=I_n(a)\). Moreover, fix a support set  $S=\{j:a_j\neq 0\}$   with \(m:=|S|\ge 2\). On the corresponding support stratum, the integral is absolutely convergent and we can differentiate under the integral sign. Indeed, 
\[
\left|\prod_{j=1}^n\frac{1}{1+ita_j}\right|= \prod_{j=1}^n\frac{1}{\sqrt{1+t^2a_j^2}}.
\]
Since $m\geq 2$, the integrand decays like $t^{-m}$ as $t \to \infty$ hence $I_n(a)<\infty$ for all such $a$.  
A similar argument justifies differentiation under the integral sign.
\end{remark}

We now deal separately with the case where there exists exactly one negative coordinate.

\begin{proposition}\label{prop:one-negative-rest-positive}
Let \(n\geq 2\), and let $a_1,\ldots,a_{n-1}\geq 0$, $a_n\leq 0,$ with $\sum_{j=1}^{n}a_j=0,$ $\sum_{j=1}^{n}a_j^2=1.$
Then \(I_n\) is maximized at
\[
a=
\left(
\frac{1}{\sqrt{2}},
0,\ldots,0,
-\frac{1}{\sqrt{2}}
\right).
\]
Moreover, \(I_n\) is minimized at
\[
a=
\left(
\frac{1}{\sqrt{n(n-1)}},
\ldots,
\frac{1}{\sqrt{n(n-1)}},
-\sqrt{\frac{n-1}{n}}
\right).
\]
\end{proposition}

\begin{remark}
In the notation of the \(n\)-dimensional simplex \(\Delta^n\subset
\mathbb R^{n+1}\), the corresponding statement is obtained by replacing
\(n\) above by \(n+1\). Thus, the lower-bound conjecture for central
sections of the simplex would follow once one proves that a global
minimizer has exactly one coordinate of one sign.
\end{remark}

\begin{proof}[Proof of Proposition~\ref{prop:one-negative-rest-positive}]
Write $b:=-a_n\geq 0.$ Since $a_1,\ldots,a_{n-1}\geq 0,$ $\sum_{j=1}^{n-1}a_j=b,$ and  $\sum_{j=1}^{n-1}a_j^2+b^2=1,$. From~\eqref{f(0) contour} we get
\[
G(0)=\frac{b^{n-2}}{\prod_{j=1}^{n-1} (b+a_j)}.
\]

For the lower bound, by the AM-GM inequality,
\[
\prod_{j=1}^{n-1}(b+a_j)
\leq
\left(\frac{\sum_{j=1}^{n-1}(b+a_j)}{n-1}\right)^{n-1}
=
\left(\frac{nb}{n-1}\right)^{n-1}.
\]
Hence
\[
G(0)
\geq
\frac{b^{n-2}}{\left(\frac{nb}{n-1}\right)^{n-1}}
=
\frac{(n-1)^{n-1}}{n^{n-1}b}.
\]
Since $b^2\leq \frac{n-1}{n},$ we obtain
\[
G(0)
\geq
\left(\frac{n-1}{n}\right)^{n-\frac32}.
\]
Equality holds when
\[
a_1=\cdots=a_{n-1}
=
\frac{1}{\sqrt{n(n-1)}},
\qquad
b=\sqrt{\frac{n-1}{n}}.
\]

For the maximum, set $x_j:=\frac{a_j}{b}\geq 0,$ for $ j=1,\ldots,n-1.$
Then $\sum_{j=1}^{n-1}x_j=1$ and
\[
b^2\left(1+\sum_{j=1}^{n-1}x_j^2\right)=1.
\]
Therefore
\[
G(0)
=
\frac{1}{b\prod_{j=1}^{n-1}(1+x_j)}
=
\frac{\sqrt{1+\sum_{j=1}^{n-1}x_j^2}}
{\prod_{j=1}^{n-1}(1+x_j)}.
\]
Since $\sum_{j=1}^{n-1}x_j^2\leq 1$ and
\[
\prod_{j=1}^{n-1}(1+x_j)\geq 1+\sum_{j=1}^{n-1}x_j=2,
\]
we get
\[
G(0)\leq \frac{\sqrt2}{2}=\frac1{\sqrt2}.
\]
Equality holds when, after a permutation,
\[
x_1=1,\qquad x_2=\cdots=x_{n-1}=0,
\]
that is, when
\[
a=
\left(
\frac{1}{\sqrt2},
0,\ldots,0,
-\frac{1}{\sqrt2}
\right).
\]
\end{proof}

\begin{proof}[Proof of Theorem~\ref{thm:mu-positive-one-negative}]

\medskip
\noindent
{\bf Step 1: Basic identities.}
Set $I:=I_n(a).$
For \(i\neq j\), define
\[
K_{ij}(a):=
\int_{-\infty}^{\infty}
\frac{t^2}{(1+i t a_i)^2(1+i t a_j)^2}
\prod_{\ell\neq i,j}\frac{1}{1+i t a_\ell}\,dt,
\]
and for each \(i\) define
\[
J_i(a):=
\int_{-\infty}^{\infty}
\frac{t^2}{(1+i t a_i)^3}
\prod_{\ell\neq i}\frac{1}{1+i t a_\ell}\,dt.
\]
Then
\[
\partial_{a_i a_i}^2 I_n(a)=-2J_i(a),
\qquad
\partial_{a_i a_j}^2 I_n(a)=-K_{ij}(a)
\quad (i\neq j).
\tag{1}
\label{eq:hessian-entries}
\]

Since \(I_n\) is homogeneous of degree \(-1\), Euler's identity gives
\[
\sum_{j=1}^n a_j\,\partial_{a_j}I_n(a)=-I.
\]
Using the Lagrange equations and the constraints
\[
\sum_{j=1}^n a_j=0,
\qquad
\sum_{j=1}^n a_j^2=1,
\]
we obtain
\[
2\lambda=-I.
\tag{2}
\label{eq:lambda-minus-I}
\]

Also, \(I>0\). Indeed, if \(I=0\), then the section
\(\Delta_n\cap a^{\perp}\) has zero \((n-1)\)-volume and obviously this is not the case.

For each \(i\), one has the integration-by-parts identity
\[
2a_iJ_i(a)=2\,\partial_{a_i}I_n(a)-\sum_{m\neq i} a_m K_{im}(a).
\tag{3}
\label{eq:Ji-general}
\]
Indeed, applying \(\int_{\mathbb R}\frac{d}{dt}(\cdots)\,dt=0\) to
\[
\frac{t^2}{(1+i t a_i)^2}
\prod_{\ell\neq i}\frac{1}{1+i t a_\ell}
\]
gives exactly \eqref{eq:Ji-general}.

We shall also use the following identity. If \(a_i\neq a_j\), then
\[
K_{ij}(a)=I.
\tag{4}
\label{eq:Kij-equals-I}
\]
Indeed, subtracting the Lagrange equations for the indices \(i\) and \(j\), and using
\eqref{eq:lambda-minus-I}, gives
\[
\partial_{a_i}I_n(a)-\partial_{a_j}I_n(a)
=
-I(a_i-a_j).
\]
On the other hand, by direct differentiation,
\[
\partial_{a_i}I_n(a)-\partial_{a_j}I_n(a)
=
-(a_i-a_j)K_{ij}(a).
\]
Since \(a_i\neq a_j\), this proves \eqref{eq:Kij-equals-I}.

\medskip
\noindent
{\bf Step 2: A repeated level cannot be non-positive.}
Let \(r\) be a value appearing in \(a\) with multiplicity \(m_r\geq 2\). Now assume \(r\neq 0\). Choose two indices \(p\neq q\) such that $a_p=a_q=r.$ Then
\[
h:=e_p-e_q\in T_a\mathcal M,
\]
because
\[
\sum_{j=1}^n h_j=0,
\qquad
\langle a,h\rangle=r-r=0.
\]
Since \(a\) is a local maximum, the constrained second variation satisfies
\[
Q(h):=D^2I_n(a)[h,h]-2\lambda\|h\|^2\leq 0.
\]
Using \eqref{eq:hessian-entries} and \eqref{eq:lambda-minus-I}, one gets
\[
Q(h)=2\bigl(I-K_{pq}(a)\bigr)\leq 0.
\]
Since \(a_p=a_q=r\), write
\[
K_{rr}:=K_{pq}(a).
\]
Thus
\[
K_{rr}\geq I.
\tag{5}
\label{eq:Krr-ge-I}
\]

Now fix one index \(i\) in the \(r\)-block. Because the level \(r\) is repeated, one has
\[
J_i(a)=K_{rr}.
\]
Apply \eqref{eq:Ji-general} to this index \(i\):
\[
2rK_{rr}
=
2(2\lambda r+\mu)
-
\sum_{m\neq i}a_mK_{im}(a).
\]
For the indices \(m\neq i\) inside the \(r\)-block, the corresponding contribution is
\[
(m_r-1)rK_{rr}.
\]
For indices outside the \(r\)-block, we have \(a_m\neq r\), hence \(K_{im}(a)=I\) by
\eqref{eq:Kij-equals-I}. Since the sum of the coordinates outside the \(r\)-block equals
\[
-m_r r,
\]
we obtain
\[
2rK_{rr}
=
2(-Ir+\mu)
-
\Big((m_r-1)rK_{rr}-m_r rI\Big).
\]
Rearranging gives
\[
(m_r+1)K_{rr}
=
(m_r-2)I+\frac{2\mu}{r}.
\tag{6}
\label{eq:repeated-identity}
\]
Combining \eqref{eq:Krr-ge-I} and \eqref{eq:repeated-identity}, we get
\[
(m_r+1)I
\leq
(m_r-2)I+\frac{2\mu}{r}.
\]
Therefore
\[
\frac{2\mu}{r}\geq 3I>0.
\tag{7}
\label{eq:mu-over-r-positive}
\]
Since \(\mu>0\), this implies \(r>0\).

Thus every repeated level is positive. In particular, no negative level is repeated.

\medskip
\noindent
{\bf Step 3: There cannot be two negative coordinates.}
Assume, for a contradiction, that \(a\) has at least two negative coordinates.
Choose two of them and denote them by
\[
x<0,
\qquad
y<0.
\]
By Step 2, negative levels cannot be repeated; hence $x\neq y.$

Define
\[
S:=\sum_{\ell\notin\{x,y\}} a_\ell=-(x+y),
\qquad
T:=\sum_{\ell\notin\{x,y\}} a_\ell^2=1-x^2-y^2.
\tag{8}
\label{eq:S-T}
\]
Define a vector \(u\in\mathbb R^n\) by
\[
u_\ell=a_\ell
\qquad
\text{for every } \ell\notin\{x,y\},
\]
and
\[
u_x:=\frac{Sy-T}{x-y},
\qquad
u_y:=\frac{T-Sx}{x-y}.
\tag{9}
\label{eq:u-def}
\]
Then
\[
u_x+u_y=-S,
\qquad
xu_x+yu_y=-T.
\]
Consequently,
\[
\sum_{j=1}^n u_j=0,
\qquad
\sum_{j=1}^n a_j u_j=0.
\]
Thus $u\in T_a\mathcal M.$
 
Also \(u\neq 0\), since otherwise \(a_\ell=0\) for every \(\ell\notin\{x,y\}\), which would contradict $ x+y=0$ with \(x<0\), \(y<0\).

We now compute
\[
Q(u):=D^2I_n(a)[u,u]-2\lambda\|u\|^2.
\]

Group the coordinates different from \(x,y\) according to their levels. If \(\rho\neq0\) is a singleton level, then by \eqref{eq:Ji-general} and \eqref{eq:Kij-equals-I},
\[
J_\rho=\frac{\mu}{\rho}-\frac{I}{2}.
\tag{10}
\label{eq:J-singleton}
\]
Indeed, for an index \(i\) with \(a_i=\rho\), all other levels are distinct from \(\rho\), so
\[
2\rho J_\rho
=
2(-I\rho+\mu)-\sum_{m\neq i}a_m I
=
2(-I\rho+\mu)+\rho I
=
-I\rho+2\mu.
\]

If \(\rho>0\) is a repeated level with multiplicity \(m_\rho\geq2\), then, by
\eqref{eq:repeated-identity},
\[
(m_\rho+1)K_{\rho\rho}
=
(m_\rho-2)I+\frac{2\mu}{\rho}.
\tag{11}
\label{eq:repeated-identity-rho}
\]

Using \eqref{eq:hessian-entries}, \eqref{eq:lambda-minus-I},
\eqref{eq:Kij-equals-I}, \eqref{eq:J-singleton}, and
\eqref{eq:repeated-identity-rho}, we get the following contribution from all
coordinates different from \(x,y\):
\[
3I\sum_{\ell\notin\{x,y\}}a_\ell^2
-
2\mu\sum_{\ell\notin\{x,y\}}a_\ell.
\tag{12}
\label{eq:outside-contribution}
\]
Here zero coordinates contribute nothing.

For the two negative singleton levels \(x\) and \(y\), again by
\eqref{eq:J-singleton},
\[
J_x=\frac{\mu}{x}-\frac{I}{2},
\qquad
J_y=\frac{\mu}{y}-\frac{I}{2}.
\]
Hence their contributions are
\[
3Iu_x^2-\frac{2\mu}{x}u_x^2
\qquad\text{and}\qquad
3Iu_y^2-\frac{2\mu}{y}u_y^2.
\tag{13}
\label{eq:xy-contributions}
\]

Combining \eqref{eq:outside-contribution} and \eqref{eq:xy-contributions}, and using
the definition of \(S\) and \(T\), we obtain
\[
Q(u)
=
3I\left(T+u_x^2+u_y^2\right)
-
2\mu\left(S+\frac{u_x^2}{x}+\frac{u_y^2}{y}\right).
\]
Equivalently,
\[
Q(u)
=
3I\|u\|^2
-
2\mu
\sum_{a_j\neq0}\frac{u_j^2}{a_j}.
\tag{14}
\label{eq:Q-final}
\]

It remains to compute the last sum. Using \eqref{eq:S-T} and \eqref{eq:u-def},
\[
\sum_{a_j\neq0}\frac{u_j^2}{a_j}
=
S+\frac{(Sy-T)^2}{x(x-y)^2}
+\frac{(T-Sx)^2}{y(x-y)^2}.
\]
A direct simplification yields
\[
\sum_{a_j\neq0}\frac{u_j^2}{a_j}
=
\frac{x+y}{xy(x-y)^2}.
\tag{15}
\label{eq:u-square-over-a}
\]
Since $x,y<0,$  we have
\[
x+y<0,
\qquad
xy>0,
\qquad
(x-y)^2>0.
\]
Therefore
\[
\sum_{a_j\neq0}\frac{u_j^2}{a_j}<0.
\tag{16}
\label{eq:u-square-over-a-negative}
\]

Finally, by \eqref{eq:Q-final}, \eqref{eq:u-square-over-a-negative}, and the assumption
\(\mu>0\), we obtain $Q(u)>0.$
 
This contradicts the second-order necessary condition for a local maximum, namely $Q(u)\leq0.$ Therefore \(a\) cannot have two negative coordinates.

\medskip
\noindent
{\bf Step 4: Conclusion.}
We have shown that \(a\) has at most one negative coordinate. On the other hand,
since
\[
\sum_{j=1}^n a_j=0
\qquad\text{and}\qquad
\sum_{j=1}^n a_j^2=1,
\]
the vector \(a\) is nonzero and cannot have all coordinates nonnegative.
Hence \(a\) has at least one negative coordinate.

Therefore \(a\) has exactly one negative coordinate.
\end{proof}

\begin{corollary}\label{cor:mu-negative-or-zero}
Under the assumptions of Theorem~\ref{thm:mu-positive-one-negative}, the following hold:
\begin{enumerate}
    \item If \(\mu<0\), then \(a\) has exactly one positive coordinate.
    \item If \(\mu=0\), then \(a\) has exactly one positive coordinate and exactly one negative coordinate.
\end{enumerate}
\end{corollary}

\begin{proof}
If \(\mu<0\), apply Theorem~\ref{thm:mu-positive-one-negative} to \(-a\). The corresponding multiplier is \(-\mu>0\). Hence \(-a\) has exactly one negative coordinate, and therefore \(a\) has exactly one positive coordinate.

Assume now that \(\mu=0\). We first show that \(a\) has at most one negative coordinate. If there were two negative coordinates \(x<0\) and \(y<0\), then the argument of Step 3 in the proof of Theorem~\ref{thm:mu-positive-one-negative} applies verbatim. Indeed, repeated negative levels are impossible: if \(r<0\) is repeated, then \eqref{eq:mu-over-r-positive} gives
\[
0=\frac{2\mu}{r}\geq 3I>0,
\]
a contradiction. Thus we may choose \(x\neq y\). The same tangent vector \(u\) gives
\[
Q(u)=3I\|u\|^2>0,
\]
contradicting the second-order necessary condition for a local maximum.

Hence \(a\) has at most one negative coordinate. Applying the same argument to \(-a\), we also get that \(a\) has at most one positive coordinate. Since
\[
\sum_{j=1}^n a_j=0
\qquad\text{and}\qquad
\sum_{j=1}^n a_j^2=1,
\]
both signs must occur. Therefore \(a\) has exactly one positive coordinate and exactly one negative coordinate.
\end{proof}
\begin{corollary}
Combining Theorem~\ref{thm:mu-positive-one-negative},
Corollary~\ref{cor:mu-negative-or-zero}, and
Proposition~\ref{prop:one-negative-rest-positive}, together with the
symmetric one-positive version of the proposition, we obtain Webb's bound.
\end{corollary}
\section{Sections of $\Delta^n$ with small distance from the centroid}\label{non-central sections with small ditance}

Let $a\in \mathbb{S}^n$ be fixed, and set
$\kappa:=\sum_{j=1}^{n+1} a_j.$  In view of \eqref{dist from bary}, and the previous discussions, prescribing the distance from the barycenter is equivalent to prescribing the value of $\kappa$.  More precisely, we prove the upper bound for $|\kappa|< 1/\sqrt
2$, and from~\eqref{dist from bary} this corresponds to $d< \sqrt{\frac{1}{(n+1)(2n+1)}}.$

\begin{theorem}\label{aff sec thm}
    Let \(|\kappa|<1/\sqrt2\). Then, for every
\(a\in\mathbb R^{n+1}\) satisfying  $\sum_{j=1}^{n+1}a_j=\kappa,$ $\sum_{j=1}^{n+1}a_j^2=1,$ one has
\[
        \operatorname{vol}_{n-1}(\Delta_n\cap H_a)
        \le
        \frac{\sqrt{n+1-\kappa^2}}{(n-1)!}
        \frac1{\sqrt{2-\kappa^2}}.
\]
Equality holds, up to permutation of the coordinates, at
\[
        a=
        \left(
        \frac{\kappa}{2}+\sqrt{\frac12-\frac{\kappa^2}{4}},
        \frac{\kappa}{2}-\sqrt{\frac12-\frac{\kappa^2}{4}},
        0,\ldots,0
        \right).
\]
\end{theorem}

Again, for convenience, we work with $n$ coordinates instead of $n+1$.

\begin{theorem}\label{thm:kappa-small-delta-positive}
Let
\[
I_n(a):=\int_{-\infty}^{\infty}\prod_{j=1}^n \frac{1}{1+i t a_j}\,dt,
\qquad
\mathcal{M}_{\kappa}:=\Bigl\{a\in\mathbb R^n:\ \sum_{j=1}^n a_j=\kappa,\ \sum_{j=1}^n a_j^2=1\Bigr\},
\]
and assume
\[
|\kappa|<\frac1{\sqrt2}.
\]
Let \(a\in \mathcal{M}_{\kappa}\) be a smooth local maximum of \(I_n\), and let
\(\lambda,\mu\in\mathbb R\) be the Lagrange multipliers:
\[
\partial_{a_j}I_n(a)=2\lambda a_j+\mu,
\qquad j=1,\dots,n.
\]
Set
\[
I:=I_n(a),
\qquad
c:=I+\kappa\mu,
\qquad
\delta:=(2-\kappa^2)\mu-\kappa I.
\]
Assume that
\(
\delta>0.
\)
Then \(a\) has exactly one negative coordinate.
\end{theorem}

We now deal with the case where only one negative coordinate appears. The following serves as a partial generalization of Proposition~\ref{prop:one-negative-rest-positive}.

\begin{proposition}\label{prop:one-negative-rest-positive-affine}
Let $a_1,\ldots,a_n\geq0,$ $a_{n+1}\leq0,$
with $\sum_{j=1}^{n+1}a_j=\kappa,$ $|\kappa|<\frac1{\sqrt2},$ and $\sum_{j=1}^{n+1}a_j^2=1.$ Then
\[
a\mapsto\operatorname{vol}_{n-1}(\Delta_n\cap a^{\perp})
\]
is maximized, up to permutation of the nonnegative coordinates, at
\[
\left(
\frac{\kappa}{2}+\sqrt{\frac{1}{2}-\frac{\kappa^2}{4}},
0,\ldots,0,
\frac{\kappa}{2}-\sqrt{\frac{1}{2}-\frac{\kappa^2}{4}}
\right).
\]
\end{proposition}

\begin{remark}
The case where there is exactly one positive coordinate can be treated similarly.
\end{remark}

\begin{proof}[Proof of Proposition~\ref{prop:one-negative-rest-positive-affine}]
Write $b:=-a_{n+1}\geq0.$ Then $\sum_{j=1}^{n}a_j=b+\kappa,$ and  $b^2+\sum_{j=1}^{n}a_j^2=1.$ Since \(|\kappa|<1/\sqrt2\), we necessarily have \(b>0\). From the identity~\eqref{affine sect contoured}, we obtain
\[
\operatorname{vol}_{n-1}\left(\Delta_n \cap a^{\perp} \right)
=
\frac{\sqrt{n+1-\kappa^2}}{(n-1)!}\cdot
\frac{b^{n-1}}{\prod_{j=1}^n(b+a_j)}.
\]
Let \(F(a)\) denote the quantity without the constant factor. We begin by observing that
\[
\prod_{j=1}^n(b+a_j)
\geq
b^{n-1}\left(b+\sum_{j=1}^n a_j\right)
=
b^{n-1}(2b+\kappa),
\]
where we used repeatedly the elementary inequality
\[
(b+x)(b+y)=b^2+b(x+y)+xy\geq b(b+x+y),
\qquad b,x,y\geq0.
\]
Thus,
\[
F(a)\leq \frac{1}{2b+\kappa}.
\]
Moreover,
\[
\sum_{j=1}^n a_j^2
\leq
\left(\sum_{j=1}^n a_j\right)^2
=
(b+\kappa)^2.
\]
Hence
\[
1=b^2+\sum_{j=1}^n a_j^2
\leq
b^2+(b+\kappa)^2.
\]
Equivalently,
\[
(2b+\kappa)^2\geq 2-\kappa^2.
\]
Since \(2b+\kappa>0\), we get
\[
F(a)\leq \frac{1}{\sqrt{2-\kappa^2}}.
\]

Equality holds if and only if equality holds in all the above estimates. This forces, after a permutation of the nonnegative coordinates, $a_2=\cdots=a_n=0,$ and $b^2+(b+\kappa)^2=1.$

Solving for \(b\), we get
\[
b=-\frac{\kappa}{2}+\sqrt{\frac12-\frac{\kappa^2}{4}}.
\]
Thus
\[
a_{n+1}=-b
=
\frac{\kappa}{2}-\sqrt{\frac12-\frac{\kappa^2}{4}},
\]
and
\[
a_1=b+\kappa
=
\frac{\kappa}{2}+\sqrt{\frac12-\frac{\kappa^2}{4}}.
\]
This proves the claim.
\end{proof}

The proof of Theorem~\ref{thm:kappa-small-delta-positive} mainly leverages the argument of Theorem~\ref{thm:mu-positive-one-negative}.

\begin{proof}[Proof of Theorem~\ref{thm:kappa-small-delta-positive}]
For brevity write
\[
T_a\mathcal{M}_{\kappa}
=
\Bigl\{h\in\mathbb R^n:\ \sum_{j=1}^n h_j=0,\ \sum_{j=1}^n a_j h_j=0\Bigr\}.
\]
Since \(a\) is a constrained local maximum, the second-order necessary condition is
\[
Q(h):=D^2 I_n(a)[h,h]-2\lambda \|h\|^2\le 0
\qquad\text{for every }h\in T_a\mathcal{M}_{\kappa}.
\]

\medskip
\noindent
{\bf Step 1: Basic identities.}
For \(i\neq j\), define
\[
K_{ij}(a):=
\int_{-\infty}^{\infty}
\frac{t^2}{(1+i t a_i)^2(1+i t a_j)^2}
\prod_{\ell\neq i,j}\frac{1}{1+i t a_\ell}\,dt,
\]
and for each \(i\) define
\[
J_i(a):=
\int_{-\infty}^{\infty}
\frac{t^2}{(1+i t a_i)^3}
\prod_{\ell\neq i}\frac{1}{1+i t a_\ell}\,dt.
\]
Then
\begin{equation}
\partial_{a_i a_i}^2 I_n(a)=-2J_i(a),
\qquad
\partial_{a_i a_j}^2 I_n(a)=-K_{ij}(a)
\quad (i\neq j).
\label{eq:kappa-small-hessian}
\end{equation}

Since \(I_n\) is homogeneous of degree \(-1\), Euler's identity gives
\[
\sum_{j=1}^n a_j\,\partial_{a_j} I_n(a)=-I.
\]
Using the Lagrange equations and the constraints
\[
\sum_{j=1}^n a_j=\kappa,
\qquad
\sum_{j=1}^n a_j^2=1,
\]
we obtain
\begin{equation}
2\lambda=-c,
\label{eq:kappa-small-lambda}
\end{equation}
where \(c=I+\kappa\mu\).

Also \(I>0\) and
for each \(i\) one has the integration-by-parts identity
\begin{equation}
2a_i J_i(a)
=
2\,\partial_{a_i}I_n(a)
-
\sum_{m\neq i} a_m K_{im}(a).
\label{eq:kappa-small-Ji-general}
\end{equation}
Indeed, apply \(\int_{\mathbb R}\frac{d}{dt}(\cdots)\,dt=0\) to
\[
\frac{t^2}{(1+i t a_i)^2}\prod_{\ell\neq i}\frac{1}{1+i t a_\ell}.
\]

\medskip
\noindent
{\bf Step 2: Distinct levels and repeated levels.}
If \(a_i\neq a_j\), then subtracting the Lagrange equations for \(i\) and \(j\) gives
\[
\partial_{a_i}I_n(a)-\partial_{a_j}I_n(a)
=
2\lambda(a_i-a_j)
=
-c(a_i-a_j),
\]
by \eqref{eq:kappa-small-lambda}. On the other hand,
\[
\partial_{a_i}I_n(a)-\partial_{a_j}I_n(a)
=
-(a_i-a_j)K_{ij}(a).
\]
Hence
\begin{equation}
K_{ij}(a)=c
\qquad\text{whenever }a_i\neq a_j.
\label{eq:kappa-small-distinct-K}
\end{equation}

Now let \(r\neq0\) be a level of multiplicity \(m_r\ge 2\), and choose \(p\neq q\) in the \(r\)-block. Since
\[
h:=e_p-e_q\in T_a\mathcal{M}_{\kappa},
\]
the second-order condition gives \(Q(h)\le 0\).

Because \(J_p(a)=J_q(a)=K_{pq}(a)\) when \(a_p=a_q=r\), \eqref{eq:kappa-small-hessian} and
\eqref{eq:kappa-small-lambda} yield
\[
Q(h)=2\bigl(c-K_{pq}(a)\bigr)\le 0.
\]
Writing
\(
K_{rr}:=K_{pq}(a),
\)
we obtain
\begin{equation}
K_{rr}\ge c.
\label{eq:kappa-small-Krr-ge-c}
\end{equation}

Fix now one index \(i\) in the \(r\)-block. Since \(J_i(a)=K_{rr}\) and \(K_{im}(a)=c\) for every \(m\) outside the \(r\)-block,
\eqref{eq:kappa-small-Ji-general} gives
\[
2rK_{rr}
=
2(2\lambda r+\mu)
-
\Bigl((m_r-1)rK_{rr}+c(\kappa-m_r r)\Bigr).
\]
Using \eqref{eq:kappa-small-lambda}, namely \(2\lambda=-c\), this becomes
\[
2rK_{rr}
=
\delta +(m_r-2)cr-(m_r-1)rK_{rr}.
\]
Thus
\begin{equation}
(m_r+1)K_{rr}
=
(m_r-2)c+\frac{\delta}{r}.
\label{eq:kappa-small-repeated}
\end{equation}
Combining \eqref{eq:kappa-small-Krr-ge-c} and \eqref{eq:kappa-small-repeated}, we get
\begin{equation}
\frac{\delta}{r}\ge 3c.
\label{eq:kappa-small-delta-over-r}
\end{equation}

Finally, if \(a_i\) is a singleton nonzero level, then \eqref{eq:kappa-small-Ji-general} together with
\eqref{eq:kappa-small-distinct-K} yields
\[
2a_i J_i(a)
=
2(2\lambda a_i+\mu)-c(\kappa-a_i)
=
-ca_i+\delta.
\]
Hence
\begin{equation}
J_i(a)=\frac{\delta}{2a_i}-\frac c2
\qquad\text{whenever }a_i\neq0\text{ is a singleton level.}
\label{eq:kappa-small-singleton-J}
\end{equation}
\noindent
{\bf Step 3: The case \(\kappa\ge 0\).}
Assume first that
\[
0\le \kappa<\frac1{\sqrt2}.
\]
From
\(
\delta=(2-\kappa^2)\mu-\kappa I
\)
one obtains
\begin{equation}
(2-\kappa^2)c=2I+\kappa\delta.
\label{eq:kappa-small-c-identity}
\end{equation}
Since \(I>0\), \(\delta>0\), and \(2-\kappa^2>0\), we get
\(
c>0.
\)

\smallskip
\noindent
{\it Step 3a: Every repeated level is positive.}
By \eqref{eq:kappa-small-delta-over-r} and \(c,\delta>0\),
\[
\frac{\delta}{r}\ge 3c>0.
\]
Hence
\(
r>0.
\)
Therefore every repeated level is positive; in particular, every negative level is a singleton.

\smallskip
\noindent
{\it Step 3b: There cannot be three negative coordinates.}
Suppose that
\[
a_i<0,\qquad a_j<0,\qquad a_k<0.
\]
By Step 3a these three levels are pairwise distinct, so \eqref{eq:kappa-small-distinct-K} gives
\[
K_{ij}(a)=K_{jk}(a)=K_{ki}(a)=c.
\]
Define
\[
\alpha:=a_j-a_k,
\qquad
\beta:=a_k-a_i,
\qquad
\gamma:=a_i-a_j,
\]
and
\[
h:=\alpha e_i+\beta e_j+\gamma e_k.
\]
Then
\[
\alpha+\beta+\gamma=0,
\qquad
a_i\alpha+a_j\beta+a_k\gamma=0,
\]
so
\[
h\in T_a\mathcal{M}_{\kappa}.
\]
Using \eqref{eq:kappa-small-hessian}, \eqref{eq:kappa-small-lambda},
\eqref{eq:kappa-small-distinct-K}, and \eqref{eq:kappa-small-singleton-J}, we obtain
\[
Q(h)
=
3c\|h\|^2
-
\delta\left(\frac{\alpha^2}{a_i}+\frac{\beta^2}{a_j}+\frac{\gamma^2}{a_k}\right).
\]
Since \(a_i,a_j,a_k<0\), the parenthesis is negative. Because \(c>0\) and \(\delta>0\), we conclude
\[
Q(h)>0,
\]
contradicting the second-order necessary condition. Thus there are not three negative coordinates.

\smallskip
\noindent
{\it Step 3c: There cannot be exactly two negative coordinates.}
Assume now that \(a\) has exactly two negative coordinates,
\[
x=-p<0,
\qquad
y=-q<0,
\qquad
p,q>0,
\]
and all remaining coordinates are nonnegative. By Step 3a,
\(
p\neq q.
\)
Let
\[
S:=\sum_{a_\ell\ge0} a_\ell = p+q+\kappa,
\qquad
T:=\sum_{a_\ell\ge0} a_\ell^2 = 1-p^2-q^2.
\]
Define \(u\in\mathbb R^n\) by
\[
u_\ell=a_\ell\quad (a_\ell\ge0),
\qquad
u_x:=\frac{Sy-T}{x-y},
\qquad
u_y:=\frac{T-Sx}{x-y}.
\]
Then
\[
u_x+u_y=-S,
\qquad
xu_x+yu_y=-T.
\]
Hence
\[
u\in T_a\mathcal{M}_{\kappa}.
\]

Exactly as in the two-negative computation in the case \(\kappa=0\), but with \(I\) replaced by \(c\)
and \(2\mu\) replaced by \(\delta\), one gets
\begin{equation}
Q(u)=3c\|u\|^2-\delta\sum_{a_j\neq0} \frac{u_j^2}{a_j}.
\label{eq:kappa-small-Q-u-general}
\end{equation}
Now
\[
\sum_{a_j\neq0} \frac{u_j^2}{a_j}
=
S-\frac{u_x^2}{p}-\frac{u_y^2}{q}.
\]
Since \(u_x+u_y=-S\), Cauchy's inequality gives
\[
(p+q)\left(\frac{u_x^2}{p}+\frac{u_y^2}{q}\right)
\ge
(u_x+u_y)^2
=
S^2.
\]
Because \(\kappa\ge0\), we have
\[
S=p+q+\kappa\ge p+q.
\]
Moreover \(S>0\), and the equality case would force \(\kappa=0\). In all cases,
\[
\frac{u_x^2}{p}+\frac{u_y^2}{q}
\ge
\frac{S^2}{p+q}
\ge
S.
\]
If equality held throughout, then \(\kappa=0\) and \(u_x/p=u_y/q\), which together with \(u_x+u_y=-S=-(p+q)\) would imply \(u_x=-p\), \(u_y=-q\). Then \(xu_x+yu_y=p^2+q^2\), contradicting \(xu_x+yu_y=-T<0\). Hence the inequality is strict:
\[
\frac{u_x^2}{p}+\frac{u_y^2}{q}>S.
\]
Therefore
\[
\sum_{a_j\neq0} \frac{u_j^2}{a_j}<0.
\]
Since \(c>0\) and \(\delta>0\), \eqref{eq:kappa-small-Q-u-general} yields
\[
Q(u)>0,
\]
again a contradiction. Hence there are not exactly two negative coordinates.

\smallskip
\noindent
{\it Step 3d: Conclusion in the easy case.}
By Steps 3b and 3c, there is at most one negative coordinate.
On the other hand, not all coordinates can be nonnegative, for otherwise
\[
1=\sum_{j=1}^n a_j^2
\le
\Bigl(\sum_{j=1}^n a_j\Bigr)^2
=
\kappa^2<1,
\]
a contradiction. Hence there is at least one negative coordinate.
Therefore, in the case \(\kappa\ge0\), there is exactly one negative coordinate.

\medskip
\noindent
{\bf Step 4: The mixed case \(\kappa<0\).}
Assume now that $-\frac1{\sqrt2}<\kappa<0,$ and write $k:=-\kappa\in \Bigl(0,\frac1{\sqrt2}\Bigr).$
 
From~\eqref{eq:kappa-small-c-identity}, we get
\[
(2-k^2)c=2I-k\delta.
\]
Hence
\begin{equation}
c>-\frac{k}{2-k^2}\,\delta.
\label{eq:kappa-small-c-lower}
\end{equation}
Since
\[
\frac{3k}{2-k^2}<\sqrt2
\]
\eqref{eq:kappa-small-c-lower} implies
\begin{equation}
3c+\sqrt2\,\delta>0.
\label{eq:kappa-small-3c+sqrt2delta}
\end{equation}

\smallskip
\noindent
{\it Step 4a: There is no repeated negative level.}
Suppose that \(r<0\) is a repeated level, of multiplicity \(m_r\ge2\).
Then
\[
m_r r^2\le \sum_{j=1}^n a_j^2=1,
\]
so
\[
|r|\le \frac1{\sqrt{m_r}}\le \frac1{\sqrt2}.
\]
Since \(\delta>0\) and \(r<0\),
\[
\frac{\delta}{r}\le -\sqrt2\,\delta.
\]
But by \eqref{eq:kappa-small-delta-over-r} and \eqref{eq:kappa-small-3c+sqrt2delta},
\[
\frac{\delta}{r}\ge 3c>-\sqrt2\,\delta,
\]
a contradiction. Hence there is no repeated negative level.

\smallskip
\noindent
{\it Step 4b: There cannot be three negative coordinates.}
Assume that
\[
a_i=-p<0,
\qquad
a_j=-q<0,
\qquad
a_k=-r<0,
\qquad
p,q,r>0.
\]
By Step 4a these are pairwise distinct, so \eqref{eq:kappa-small-distinct-K} applies.
Let
\[
\alpha:=q-r,
\qquad
\beta:=r-p,
\qquad
\gamma:=p-q,
\]
and
\[
h:=\alpha e_i+\beta e_j+\gamma e_k.
\]
Then
\[
\alpha+\beta+\gamma=0,
\qquad
a_i\alpha+a_j\beta+a_k\gamma=0,
\]
so $h\in T_a\mathcal M_\kappa.$

As before,
\[
Q(h)
=
3c\|h\|^2
+
\delta\left(
\frac{(q-r)^2}{p}
+
\frac{(r-p)^2}{q}
+
\frac{(p-q)^2}{r}
\right).
\]

\begin{claim}
It holds that
\begin{equation}
\frac{(q-r)^2}{p}
+
\frac{(r-p)^2}{q}
+
\frac{(p-q)^2}{r}
\ge
\sqrt2\bigl((q-r)^2+(r-p)^2+(p-q)^2\bigr).
\label{eq:kappa-small-three-neg-sharp}
\end{equation}
\end{claim}

\begin{proof}[Proof of Claim]
Set
\[
A:=(q-r)^2,
\qquad
B:=(r-p)^2,
\qquad
C:=(p-q)^2.
\]
By the Cauchy-Schwarz inequality
\[
\frac{A}{p}+\frac{B}{q}+\frac{C}{r}
\ge
\frac{(A+B+C)^2}{pA+qB+rC}.
\]
Also,
\[
(pA+qB+rC)^2
\le
(p^2+q^2+r^2)(A^2+B^2+C^2)
\le
A^2+B^2+C^2,
\]
because \(p^2+q^2+r^2\le1\).
Now let
\[
x:=q-r,
\qquad
y:=r-p,
\qquad
z:=p-q.
\]
Then
\[
x+y+z=0,
\qquad
A=x^2,\quad B=y^2,\quad C=z^2.
\]
Therefore
\[
A^2+B^2+C^2
=
x^4+y^4+z^4
=
\frac12(x^2+y^2+z^2)^2
=
\frac12(A+B+C)^2.
\]
Hence
\[
pA+qB+rC
\le
\frac{A+B+C}{\sqrt2},
\]
and \eqref{eq:kappa-small-three-neg-sharp} follows.
\end{proof}

Using \eqref{eq:kappa-small-three-neg-sharp}, we get
\[
Q(h)
\ge
(3c+\sqrt2\,\delta)\|h\|^2>0
\]
by \eqref{eq:kappa-small-3c+sqrt2delta}. This contradicts the second-order necessary condition. Hence there cannot be three negative coordinates.

\smallskip
\noindent
{\it Step 4c: There cannot be exactly two negative coordinates.}
Assume now that \(a\) has exactly two negative coordinates,
\[
x=-p<0,
\qquad
y=-q<0,
\qquad
p,q>0,
\]
and all the others are nonnegative. By Step 4a, $p\neq q.$ Set $s:=p+q,$ and $z:=(p-q)^2.$

Then
\[
k<s<\sqrt2,
\qquad
0<z<\min\{s^2,\,2-s^2\},
\]
because \(p+q>k\), as the sum of the nonnegative coordinates equals \(p+q-k>0\), and
\[
p^2+q^2=\frac{s^2+z}{2}<1.
\]

Let
\[
S:=\sum_{a_\ell\ge0} a_\ell=s-k,
\qquad
T:=\sum_{a_\ell\ge0} a_\ell^2=1-p^2-q^2.
\]
Define \(u\in\mathbb R^n\) by
\[
u_\ell=a_\ell\quad (a_\ell\ge0),
\qquad
u_x:=\frac{Sy-T}{x-y},
\qquad
u_y:=\frac{T-Sx}{x-y}.
\]
Then $u\in T_a\mathcal{M}_{\kappa},$ and, as in \eqref{eq:kappa-small-Q-u-general},
\[
Q(u)=3c\|u\|^2-\delta\sum_{a_j\neq0} \frac{u_j^2}{a_j}.
\]

\begin{claim}
We have
\begin{equation}
-\sum_{a_j\neq0} \frac{u_j^2}{a_j}\ge \sqrt2\,\|u\|^2.
\label{eq:kappa-small-two-neg-sharp}
\end{equation}
\end{claim}

\begin{proof}[Proof of Claim]
We first record one additional consequence of the constraints. Recall that $s=p+q,$ $z=(p-q)^2,$ and that the sum of the nonnegative coordinates is $S=s-k>0,$ whereas their sum of squares is $T=1-p^2-q^2.$ Since the nonnegative coordinates are nonnegative, we have $T\le S^2<s^2.$ On the other hand, $p^2+q^2<s^2.$ Therefore $1=p^2+q^2+T<2s^2,$ and hence
\begin{equation}
s>\frac1{\sqrt2}.
\label{eq:kappa-small-s-lower}
\end{equation}
Thus the admissible range is
\[
0<k<\frac1{\sqrt2},
\qquad
\frac1{\sqrt2}<s<\sqrt2,
\qquad
0<z<\min\{s^2,2-s^2\}.
\]

A direct computation gives
\begin{equation}
\|u\|^2
=
\frac{k^2s^2+k^2z-4ks-2z+4}{2z},
\label{eq:kappa-small-u-norm-corrected}
\end{equation}
and
\begin{equation}
-\sum_{a_j\neq0} \frac{u_j^2}{a_j}
-\sqrt2\,\|u\|^2
=
\frac{\widetilde P_{k,s}(z)}{2z(s^2-z)},
\label{eq:kappa-small-Ptilde}
\end{equation}
where
\begin{align}
\widetilde P_{k,s}(z)
={}&
(\sqrt2\,k^2+2k-2\sqrt2)z^2
\notag\\
&+
(6k^2s-2ks^2-4\sqrt2\,ks-8k+2\sqrt2\,s^2+4\sqrt2)z
\notag\\
&-\sqrt2\,k^2s^4
+2k^2s^3
+4\sqrt2\,ks^3
-8ks^2
-4\sqrt2\,s^2
+8s.
\label{eq:kappa-small-Ptilde-def}
\end{align}
Since $2z(s^2-z)>0,$ it is enough to prove that $\widetilde P_{k,s}(z)>0$ throughout the admissible range.

We first show that, for fixed \(s,z\), the polynomial
\(\widetilde P_{k,s}(z)\) is decreasing in \(k\). Differentiating
\eqref{eq:kappa-small-Ptilde-def} with respect to \(k\), we get
\begin{equation}
\frac{\partial}{\partial k}\widetilde P_{k,s}(z)
=
-2M_{k,s}(z),
\label{eq:kappa-small-Ptilde-derivative}
\end{equation}
where
\begin{align}
M_{k,s}(z)
={}&
-(\sqrt2\,k+1)z^2
+
(s^2+2\sqrt2\,s+4-6ks)z
\notag\\
&+
s^2(\sqrt2\,ks^2-2ks-2\sqrt2\,s+4).
\label{eq:kappa-small-Mtilde}
\end{align}
For fixed \(k,s\), the polynomial \(M_{k,s}\) is concave in \(z\).
Therefore it is enough to check its positivity at the endpoints of the
admissible interval.

Assume first that
\[
\frac1{\sqrt2}<s\le 1.
\]
Then the admissible interval is \(0<z<s^2\). We have
\[
M_{k,s}(0)
=
s^2(2-\sqrt2\,s)(2-ks)>0,
\]
and
\[
M_{k,s}(s^2)
=
8s^2(1-ks)>0.
\]
Hence
\(
M_{k,s}(z)>0
\)
for $(0<z<s^2)$.
Assume next that
\[
1\le s<\sqrt2.
\]
Then the admissible interval is \(0<z<2-s^2\). Again,
\(
M_{k,s}(0)>0.
\)
At the other endpoint we get
\[
M_{k,s}(2-s^2)
=
2(\sqrt2-s)
\Big[
s^3+3\sqrt2s^2+3s+\sqrt2
-
k(2s^2+4\sqrt2s+2)
\Big].
\]
The expression in brackets is decreasing in \(k\), hence it is minimized
at \(k=1/\sqrt2\). At this endpoint it equals
\[
s(s^2+2\sqrt2\,s-1)>0.
\]
Thus
\[
M_{k,s}(2-s^2)>0,
\]
and therefore
\(
M_{k,s}(z)>0
\)
for
\((0<z<2-s^2)\).
By \eqref{eq:kappa-small-Ptilde-derivative},
\[
\frac{\partial}{\partial k}\widetilde P_{k,s}(z)<0.
\]
Since \(k<1/\sqrt2\), it follows that
\[
\widetilde P_{k,s}(z)>
\widetilde P_{1/\sqrt2,s}(z).
\]

It remains to prove positivity at \(k=1/\sqrt2\). A direct substitution gives
\begin{align}
\widetilde P_{1/\sqrt2,s}(z)
={}&
-\frac{\sqrt2}{2}z^2
+
(\sqrt2\,s^2-s)z
\notag\\
&-\frac{\sqrt2}{2}s^4
+5s^3
-8\sqrt2\,s^2
+8s.
\label{eq:kappa-small-Ptilde-edge}
\end{align}
This is a concave quadratic polynomial in \(z\). Hence its minimum on the
admissible interval is attained at an endpoint.

If
\[
\frac1{\sqrt2}<s\le 1,
\]
then the endpoints are \(0\) and \(s^2\). We compute
\[
\widetilde P_{1/\sqrt2,s}(0)
=
\frac{\sqrt2}{2}s(\sqrt2-s)(s-2\sqrt2)^2>0,
\]
and
\[
\widetilde P_{1/\sqrt2,s}(s^2)
=
4s(s-\sqrt2)^2>0.
\]
Therefore
\(
\widetilde P_{1/\sqrt2,s}(z)>0\)
for
\((0<z<s^2).
\)
If
\[
1\le s<\sqrt2,
\]
then the endpoints are \(0\) and \(2-s^2\). We already have
\[
\widetilde P_{1/\sqrt2,s}(0)>0,
\]
and
\[
\widetilde P_{1/\sqrt2,s}(2-s^2)
=
2\sqrt2(\sqrt2-s)
\left(s-\frac1{\sqrt2}\right)(s^2+1)>0.
\]
Consequently,
\[
\widetilde P_{k,s}(z)>0
\]
throughout the admissible range. By \eqref{eq:kappa-small-Ptilde}, this gives
\[
-\sum_{a_j\neq0} \frac{u_j^2}{a_j}
-\sqrt2\,\|u\|^2>0,
\]
and hence
\[
-\sum_{a_j\neq0} \frac{u_j^2}{a_j}
>
\sqrt2\,\|u\|^2.
\]
This proves \eqref{eq:kappa-small-two-neg-sharp}.
\end{proof}

By the claim,
\[
Q(u)\ge (3c+\sqrt2\,\delta)\|u\|^2>0
\]
by \eqref{eq:kappa-small-3c+sqrt2delta}. This contradicts the second-order necessary condition. Thus the case of exactly two negative coordinates is impossible.

\smallskip
\noindent
{\it Step 4d: Conclusion in the mixed case.}
By Steps 4b and 4c, there is at most one negative coordinate.
Since \(\kappa<0\), at least one negative coordinate must exist.
Therefore, in the case \(\kappa<0\), there is exactly one negative coordinate.

\medskip
\noindent
Combining Steps 3 and 4, we conclude that under the assumptions
\[
|\kappa|<\frac1{\sqrt2},
\qquad
\delta>0,
\]
the local maximizer \(a\) has exactly one negative coordinate.
\end{proof}

\begin{corollary}\label{cor:kappa-small-delta-negative}
Let
\[
I_n(a):=\int_{-\infty}^{\infty}\prod_{j=1}^n \frac{1}{1+i t a_j}\,dt,
\qquad
\mathcal{M}_{\kappa}:=\Bigl\{a\in\mathbb R^n:\ \sum_{j=1}^n a_j=\kappa,\ \sum_{j=1}^n a_j^2=1\Bigr\},
\]
and assume $|\kappa|<\frac1{\sqrt2}.$
 
Let \(a\in\mathcal M_\kappa\) be a smooth local maximum of \(I_n\), and let
\(\lambda,\mu\in\mathbb R\) be the Lagrange multipliers:
\[
\partial_{a_j}I_n(a)=2\lambda a_j+\mu.
\]
Set
\[
I:=I_n(a),
\qquad
\delta:=(2-\kappa^2)\mu-\kappa I.
\]
If $\delta<0,$ then \(a\) has exactly one positive coordinate.
\end{corollary}

\begin{proof}
Set $b:=-a.$ Then $\sum_{j=1}^n b_j=-\kappa,$ $\sum_{j=1}^n b_j^2=1,$ and $|-\kappa|=|\kappa|<\frac1{\sqrt2}.$
 Moreover, $I_n(b)=I_n(a),$ because the change of variables \(t\mapsto -t\) leaves the integral invariant. If
\[
\partial_{a_j}I_n(a)=2\lambda a_j+\mu,
\]
then
\[
\partial_{b_j}I_n(b)=2\lambda b_j-\mu.
\]
Thus the new multiplier \(\mu\) is \(-\mu\), while \(I_n\) is unchanged and \(\kappa\) is replaced by \(-\kappa\). Therefore
\[
\delta(b)
=
(2-\kappa^2)(-\mu)-(-\kappa)I_n(a)
=
-\bigl((2-\kappa^2)\mu-\kappa I_n(a)\bigr)
=
-\delta(a)>0.
\]
Applying Theorem~\ref{thm:kappa-small-delta-positive} to \(b\), we conclude that \(b\) has exactly one negative coordinate.
Equivalently, \(a\) has exactly one positive coordinate.
\end{proof}
\begin{proof}[Proof of Theorem~\ref{aff sec thm}]
Let \(a\) be a global maximizer and set
\[
I:=I_n(a),
\qquad
\delta:=(2-\kappa^2)\mu-\kappa I.
\]
We distinguish three cases.

If \(\delta>0\), then Theorem~\ref{thm:kappa-small-delta-positive} shows that
\(a\) has exactly one negative coordinate. Hence
Proposition~\ref{prop:one-negative-rest-positive-affine} gives
\[
\operatorname{vol}_{n-1}(\Delta_n\cap H_a)
\le
\frac{\sqrt{n+1-\kappa^2}}{(n-1)!}
\frac1{\sqrt{2-\kappa^2}}.
\]

If \(\delta<0\), then Corollary~\ref{cor:kappa-small-delta-negative} shows that
\(a\) has exactly one positive coordinate. Applying
Proposition~\ref{prop:one-negative-rest-positive-affine} to \(-a\), with
\(-\kappa\) in place of \(\kappa\), gives the same upper bound.

It remains to consider the case \(\delta=0\). We claim that in this case \(a\)
has exactly one negative coordinate. Indeed, from
\[
(2-\kappa^2)c=2I+\kappa\delta
\]
we get
\[
c=\frac{2I}{2-\kappa^2}>0.
\]
Now let \(r\neq0\) be a repeated level. The derivation of
\eqref{eq:kappa-small-delta-over-r} gives
\[
\frac{\delta}{r}\ge 3c.
\]
Since \(\delta=0\) and \(c>0\), this is impossible. Hence no nonzero level
can be repeated.

Suppose that there are three negative coordinates. Then they are pairwise
distinct. Repeating the argument of Step 3b in the proof of
Theorem~\ref{thm:kappa-small-delta-positive}, with \(\delta=0\), gives a
tangent vector \(h\) such that
\[
Q(h)=3c\|h\|^2>0,
\]
contradicting the second-order necessary condition.

Suppose next that there are exactly two negative coordinates. Define the
tangent vector \(u\) as in Step 3c in the proof of
Theorem~\ref{thm:kappa-small-delta-positive}. Since \(|\kappa|<1\), not all
coordinates can be nonpositive; hence \(u\neq0\). The computation
\eqref{eq:kappa-small-Q-u-general}, with \(\delta=0\), gives
\[
Q(u)=3c\|u\|^2>0,
\]
again a contradiction. Thus \(a\) has at most one negative coordinate.

Finally, \(a\) must have at least one negative coordinate. Otherwise all
coordinates would be nonnegative, and then
\[
1=\sum_{j=1}^{n}a_j^2
\le
\left(\sum_{j=1}^{n}a_j\right)^2
=
\kappa^2<1,
\]
a contradiction. Therefore \(a\) has exactly one negative coordinate.

Thus, also in the case \(\delta=0\), Proposition~\ref{prop:one-negative-rest-positive-affine}
applies and gives
\[
\operatorname{vol}_{n-1}(\Delta_n\cap H_a)
\le
\frac{\sqrt{n+1-\kappa^2}}{(n-1)!}
\frac1{\sqrt{2-\kappa^2}}.
\]

The equality case follows from the equality case in
Proposition~\ref{prop:one-negative-rest-positive-affine}, up to permutation
of the coordinates.
\end{proof}

\section{Restricted sections of $B_1^n$}\label{sec: sec B_1}
A central hyperplane section of $B_1^n$ passing through the barycenter of a facet is characterized by two constraints on the normal vector $a$. Without loss of generality, we consider the facet $F_+=\operatorname{conv}\{e_1,\ldots,e_n\},$ whose barycenter is  $f_+=\left(\frac1n,\ldots,\frac1n\right).$ The condition \(a^\perp\) passes through \(f_+\) is equivalent to $\sum_{j=1}^n a_j=0.$
Accordingly, we consider the set
$$
\mathcal{M}:=\left\{a\in\mathbb{R}^n:\ \sum_{r=1}^n a_r^2=1,\ \sum_{r=1}^n a_r=0\right\}.
$$

Our main theorem is the following.

\begin{theorem}\label{thm:section B_1^n centroid}
The function
$$
a\in \mathcal{M}\mapsto \operatorname{vol}_{n-1}(B_1^n\cap a^\perp)
$$
attains its minimum, when $n$ is even, at the half-plus/half-minus vector
$$
\left(
\underbrace{\frac{1}{\sqrt{n}},\ldots,\frac{1}{\sqrt{n}}}_{\frac{n}{2}\text{ times}},
\underbrace{-\frac{1}{\sqrt{n}},\ldots,-\frac{1}{\sqrt{n}}}_{\frac{n}{2}\text{ times}}
\right),
$$
and, when $n$ is odd, at the $\bigl(\frac{n-1}{2},\frac{n+1}{2}\bigr)$-split vector
$$
\left(
\underbrace{\sqrt{\frac{n+1}{n(n-1)}},\ldots,\sqrt{\frac{n+1}{n(n-1)}}}_{\frac{n-1}{2}\text{ times}},
\underbrace{-\sqrt{\frac{n-1}{n(n+1)}},\ldots,-\sqrt{\frac{n-1}{n(n+1)}}}_{\frac{n+1}{2}\text{ times}}
\right).
$$
\end{theorem}
\begin{conjecture}
The function
\[
        a\in M\longmapsto \operatorname{vol}_{n-1}(B_1^n\cap a^\perp)
\]
attains its maximum for $n<6$ at
$$
\left(\frac{1}{\sqrt{2}},-\frac{1}{\sqrt{2}},0,\ldots,0\right),
$$
whereas for $n\geq 6$ it is maximized at
$$
\left(
\underbrace{\frac{1}{\sqrt{n(n-1)}},\ldots,\frac{1}{\sqrt{n(n-1)}}}_{n-1\text{ times}},
-\sqrt{\frac{n-1}{n}}
\right).
$$  
\end{conjecture}
\begin{remark}
The $p$-R\'enyi entropy of a continuous random variable $X$ with density $f$ is defined by
$$
h_p(X)=\frac{1}{1-p}\log\left(\int_{\mathbb{R}} f(x)^p\,dx\right),
$$
for $p\in(0,\infty)\setminus\{1\}$. As usual, one recovers the Shannon entropy in the limit $p\to 1$, namely $h_1(X)=h(X)$, while
$$
h_\infty(X)=-\log \|f\|_\infty.
$$

A simple observation is that if $U$ and $V$ are i.i.d.\ log-concave random variables, then
$$
h_2(U)=h_\infty(U-V).
$$
Indeed, if $f$ denotes the density of $U$, then
$$
h_2(U)=-\log\left(\int_{\mathbb{R}} f(x)^2\,dx\right),
$$
whereas the density of $U-V$ is $f\ast \widetilde f$, with $\widetilde f(x)=f(-x)$, and its maximum is attained at $0$; hence
$$
\|f\ast \widetilde f\|_\infty=(f\ast \widetilde f)(0)=\int_{\mathbb{R}} f(x)^2\,dx.
$$

Now let $a_1,\ldots,a_n\in\mathbb{R}$ satisfy $\sum_{j=1}^n a_j^2=1$, and let $\mathcal E_1,\ldots,\mathcal E_n$ be i.i.d.\ standard exponential random variables. Then
\begin{equation}\label{h_2=h_{infty}}
h_2\left(\sum_{j=1}^n a_j\mathcal E_j\right)
=
h_\infty\left(\sum_{j=1}^n a_j(\mathcal E_j-\widetilde{\mathcal E}_j)\right),
\end{equation}
where $\widetilde{\mathcal E}_j$ denotes an independent copy of $\mathcal E_j$. Setting
$$
V_j:=\mathcal E_j-\widetilde{\mathcal E}_j,
$$
we see that $V_1,\ldots,V_n$ are i.i.d.\ standard Laplace random variables. Since Laplace random variables are Gaussian  mixtures, one obtains sharp bounds from the corresponding comparison principles; see \cite{eskenazis2018gaussian}. In fact, the same argument applies more generally to any sequence $X_1,\ldots,X_n$ of i.i.d.\ random variables such that $X_j-\widetilde X_j$ is a Gaussian  mixture, for example when $X_j\sim \Gamma(\gamma)$ with $\gamma>0$.

Assume now in addition that
$$
\sum_{j=1}^n a_j=0
\qquad\text{and}\qquad
\sum_{j=1}^n a_j^2=1.
$$
Let $G$ denote the density of $\sum_{j=1}^n a_jV_j$. By \eqref{h_2=h_{infty}} and the fact that $\sum_{j=1}^n a_jV_j$ is symmetric and log-concave, its density attains its maximum at $0$, and therefore
$$
h_2\left(\sum_{j=1}^n a_j\mathcal E_j\right)
=
h_\infty\left(\sum_{j=1}^n a_jV_j\right)
=
-\log \|G\|_\infty
=
-\log G(0).
$$
To estimate $G(0)$, we use the Fourier inversion formula:
$$
G(0)=\frac{1}{2\pi}\int_{\mathbb{R}}\prod_{j=1}^{n}\frac{1}{1+a_j^2t^2}\,dt.
$$
Thus, Theorem~\ref{thm:section B_1^n centroid} yields sharp bounds in this setting under the two constraints above.
\end{remark}

We remind the formula~\eqref{sections of B_1^n}
\begin{equation*}
    \operatorname{vol}_{n-1}(B_1^n\cap a^{\perp})
    =
    \frac{2^n}{\pi (n-1)!}
    \int_{0}^{\infty}\prod_{j=1}^n \frac{1}{1+a_j^2 t^2}\,dt.
\end{equation*}
We proceed to show the bound of Theorem~\ref{thm:section B_1^n centroid} pointwise. Let
$$
G_t(a_1,\ldots,a_n)=\prod_{j=1}^n \frac{1}{1+a_j^2t^2}.
$$
Our goal is to show that, on $\mathcal M$, the function $G_t$ is minimized at the half-plus/half-minus vector when $n$ is even, and at the almost half-plus/half-minus vector when $n$ is odd. Equivalently, it suffices to show that $1/G_t$ is maximized at these configurations.

It is well known, (see~\ref{sec:2})
\begin{equation}\label{e_k(T+x)}
\frac{1}{G_t}=\sum_{k=0}^n e_k(a_1^2,\ldots,a_n^2)t^{2k}.
 \end{equation}
Therefore, it is enough to prove that, for each $k\leq n$, the elementary symmetric polynomial $e_k(a_1^2,\ldots,a_n^2)$, attains its maximum at the corresponding extremal configuration, on $\mathcal{M}$.

To this end, fix integers $n\geq 3$ and $3\leq k\leq n$, and define, for $a=(a_1,\dots,a_n)\in\mathbb{R}^n$,
$$
F(a):=e_k(a_1^2,\dots,a_n^2)=\sum_{1\leq i_1<\cdots<i_k\leq n} a_{i_1}^2\cdots a_{i_k}^2.
$$
We shall maximize $F$ over the compact constraint set
$$
\mathcal{M}=\left\{a\in\mathbb{R}^n:\ \sum_{i=1}^n a_i^2=1,\ \sum_{i=1}^n a_i=0\right\}.
$$
\begin{theorem}\label{thm: sharp max e_k}
For every $3\leq k\leq n$, the function
$$
(a_1,\ldots,a_n)\in \mathcal{M}\mapsto e_k(a_1^2,\ldots,a_n^2)
$$
attains its maximum, when $n$ is even, at the half-plus/half-minus vector
$$
\left(
\underbrace{\frac{1}{\sqrt{n}},\ldots,\frac{1}{\sqrt{n}}}_{\frac{n}{2}\text{ times}},
\underbrace{-\frac{1}{\sqrt{n}},\ldots,-\frac{1}{\sqrt{n}}}_{\frac{n}{2}\text{ times}}
\right),
$$
and, when $n$ is odd, at the $\bigl(\frac{n-1}{2},\frac{n+1}{2}\bigr)$-split vector
$$
\left(
\underbrace{\sqrt{\frac{n+1}{n(n-1)}},\ldots,\sqrt{\frac{n+1}{n(n-1)}}}_{\frac{n-1}{2}\text{ times}},
\underbrace{-\sqrt{\frac{n-1}{n(n+1)}},\ldots,-\sqrt{\frac{n-1}{n(n+1)}}}_{\frac{n+1}{2}\text{ times}}
\right).
$$
Moreover, for every $k\geq 3$, the minimum value is equal to $0$, and it is attained at
$$
\left(\frac{1}{\sqrt{2}},-\frac{1}{\sqrt{2}},0,\ldots,0\right).
$$
\end{theorem}

\begin{remark}
For $k=1$, one trivially has $F(a)=1$. For $k=2$, the classical Newton identity yields
\begin{equation}
e_2(a_1^2,\ldots,a_n^2)=\frac{1-p_4(a_1,\ldots,a_n)}{2},
\end{equation}
where
$$
p_m(a)=\sum_{j=1}^n a_j^m
$$
denotes the $m$-th power sum.

In the recent paper \cite{brazitikos2025sharp}, it was shown that $p_4$ on $\mathcal{M}$ is maximized at
$$
a=\left(-\sqrt{\frac{n-1}{n}},\frac{1}{\sqrt{n(n-1)}},\ldots,\frac{1}{\sqrt{n(n-1)}}\right).
$$
and it is minimized at the desired vector. 

This phenomenon also appears in other settings; see, for instance, \cite{rivin2002counting,pandis2026extremal} and quite recently in \cite{holevo2026conjecture,zhang2026proof}. In particular, this suggests that the minimizer of \(e_2(a_1^2,\ldots,a_n^2)\) under the same
constraints is different from $k\geq 3$. Indeed, for $k\geq 3$,
the minimum of $e_k(a^2)$ is attained at  zero.
Consequently, the desired lower estimate for the section cannot be obtained
solely from a pointwise comparison of the elementary symmetric polynomials.
\end{remark}
\begin{theorem}\label{thm:twolevel}
Let $3\le k\le n$. Every global maximizer $a^\star\in\mathcal{M}$ of $F$ has exactly two nonzero values:
there exist $\gamma_1,\gamma_2\ge 1$ with $\gamma_1+\gamma_2+\gamma_3=n$ and numbers $\alpha>0$, $\beta>0$ such that, after permuting coordinates,
\[
a^\star=(\underbrace{\alpha,\dots,\alpha}_{\gamma_1\text{ times}},\underbrace{-\beta,\dots,-\beta}_{\gamma_2\text{ times}},\underbrace{0,\ldots,0}_{\gamma_3 \,\text{times}}).
\]

\end{theorem}

Write $x_i:=a_i^2\ge 0$ and, for each $i$, define
\[
E_i:=\ek_{\,k-1}(x_1,\dots,\widehat{x_i},\dots,x_n),
\qquad
T_{ij}:=\ek_{\,k-2}(x_1,\dots,\widehat{x_i},\dots,\widehat{x_j},\dots,x_n)\quad(i\neq j).
\]
(Here $\widehat{\cdot}$ means that entry is omitted.)

\begin{lemma}\label{lem:derivatives}
For $F(a)=\ek(a_1^2,\dots,a_n^2)$ we have, for every $i$,
\[
\frac{\partial F}{\partial a_i}(a)=2a_i E_i.
\]
Moreover, the Hessian satisfies
\[
\frac{\partial^2F}{\partial a_i^2}(a)=2E_i,\qquad
\frac{\partial^2F}{\partial a_i\partial a_j}(a)=4a_i a_j\,T_{ij}\quad (i\neq j).
\]
\end{lemma}

\begin{proof}
Each monomial in $F$ contains $a_i^2$ either $0$ times or $1$ time. Differentiating gives
\[
\frac{\partial F}{\partial a_i}= \sum_{\substack{I\subset\{1,\dots,n\}\\ |I|=k,\ i\in I}} 2a_i\prod_{\ell\in I\setminus\{i\}} a_\ell^2
=2a_i\,\ek_{k-1}(x_{\widehat{i}})=2a_i E_i.
\]
Differentiating again yields $\partial_{ii}^2F=2E_i$, and for $i\neq j$,
\[
\partial_{ij}^2F=\frac{\partial}{\partial a_j}(2a_iE_i)
=2a_i\cdot \frac{\partial}{\partial a_j}\ek_{k-1}(x_{\widehat{i}})
=2a_i\cdot 2a_j\,\ek_{k-2}(x_{\widehat{i},\widehat{j}})
=4a_i a_j T_{ij}.
\]
\end{proof}

\begin{lemma}\label{lem:diff}
For $i\neq j$,
\[
E_i-E_j=(x_j-x_i)\,T_{ij}=(a_j^2-a_i^2)\,T_{ij}.
\]
\end{lemma}

\begin{proof}
Expand $E_i=\ek_{k-1}(x_{\widehat{i}})$ and $E_j=\ek_{k-1}(x_{\widehat{j}})$ as sums over $(k-1)$-subsets.
Terms that avoid both $i$ and $j$ cancel in the difference. The remaining terms are exactly those that use $x_j$
but not $x_i$ (in $E_i$) minus those that use $x_i$ but not $x_j$ (in $E_j$), hence
\[
E_i-E_j=(x_j-x_i)\,\ek_{k-2}(x_{\widehat{i},\widehat{j}})=(x_j-x_i)\,T_{ij}.
\]
\end{proof}

Let $a^\star\in\mathcal{M}$ be a global maximizer. Let $I:=\{i:\ a_i^\star\neq 0\}$ be the support.
We may apply Lagrange multipliers on the smooth manifold obtained by freezing the zero coordinates (if any),
so the following relations hold for all $i\in I$.

\begin{lemma}\label{lem:lagrange}
There exist real numbers $\lambda,\mu$ such that for every $i\in I$,
\begin{equation}\label{eq:star}
2a_iE_i=\mu+2\lambda a_i.
\end{equation}
Equivalently,
\begin{equation}\label{eq:Ei-lam}
E_i-\lambda=\frac{\mu}{2a_i}\qquad (i\in I).
\end{equation}
\end{lemma}

\begin{proof}
Consider the Lagrangian
\[
\mathcal{L}(a)=F(a)-\lambda\Big(\sum_{r=1}^n a_r^2-1\Big)-\mu\Big(\sum_{r=1}^n a_r\Big).
\]
At a constrained critical point we have $\nabla F = 2\lambda a + \mu\mathbf{1}$ on the active coordinates.
Using Lemma~\ref{lem:derivatives} gives \eqref{eq:star}. If $a_i\neq 0$, divide by $2a_i$ to obtain \eqref{eq:Ei-lam}.
\end{proof}

\begin{lemma}\label{lem:pairmu}
Fix distinct $i,j\in I$ with $a_i\neq a_j$. Then
\begin{equation}\label{eq:mu-pair}
\mu=2a_ia_j(a_i+a_j)\,T_{ij}.
\end{equation}
\end{lemma}

\begin{proof}
Write \eqref{eq:star} for $i$ and $j$:
\[
2a_iE_i=\mu+2\lambda a_i,\qquad 2a_jE_j=\mu+2\lambda a_j.
\]
Multiply the first by $a_j$, the second by $a_i$, and subtract:
\[
2a_ia_j(E_i-E_j)=\mu(a_j-a_i).
\]
By Lemma~\ref{lem:diff}, $E_i-E_j=(a_j^2-a_i^2)T_{ij}=(a_j-a_i)(a_j+a_i)T_{ij}$.
Cancel the factor $(a_j-a_i)\neq 0$ to obtain \eqref{eq:mu-pair}.
\end{proof}

Let $a^\star$ be a global maximizer and let $\lambda,\mu$ be as in Lemma~\ref{lem:lagrange}.
The second-order necessary condition says that the Hessian of the Lagrangian is negative semidefinite
on the tangent space to $\mathcal{M}$ at $a^\star$.

Concretely, for any tangent vector $h\in\mathbb{R}^n$ with
\[
\sum_{r=1}^n h_r=0,\qquad \sum_{r=1}^n a_r h_r=0,
\]
we must have
\[
Q(h):=h^\top\big(\nabla^2F(a^\star)-2\lambda I\big)h\le 0,
\]
since the constraint $\sum a_r^2=1$ contributes $-2\lambda I$ to the Hessian, and the linear constraint $\sum a_r=0$
contributes nothing to the Hessian.

\begin{lemma}\label{lem:Qmu}
For any $h$ supported on indices in $I$,
\begin{equation}\label{eq:Qmu}
Q(h)=\mu\sum_{r\in I}\frac{h_r^2}{a_r}\;+\;8\sum_{\substack{i<j\\ i,j\in I}} a_ia_j\,T_{ij}\,h_i h_j.
\end{equation}
\end{lemma}

\begin{proof}
By Lemma~\ref{lem:derivatives},
\[
h^\top(\nabla^2F-2\lambda I)h=\sum_{r\in I} (2E_r-2\lambda)h_r^2+\sum_{\substack{i\neq j\\ i,j\in I}} 4a_i a_jT_{ij}h_i h_j.
\]
Rewrite the off-diagonal sum as $8\sum_{i<j}a_ia_jT_{ij}h_i h_j$.
Using \eqref{eq:Ei-lam}, we have $2(E_r-\lambda)=\mu/a_r$ for $r\in I$, hence the diagonal sum becomes
$\mu\sum_{r\in I}h_r^2/a_r$. This gives \eqref{eq:Qmu}.
\end{proof}

Now we exclude a third level. Assume, for contradiction, that $a^\star$ has \emph{three} distinct nonzero values,
with exactly one negative value and two distinct positive values. That is, there exist indices $i,j,k\in I$
and numbers $a,b,c>0$ with $b\neq c$ such that
\[
a_i=-a,\qquad a_j=b,\qquad a_k=c.
\]
Let $S:=\{1,\dots,n\}\setminus\{i,j,k\}$ and define
\[
T_0:=e_{k-2}(a_\ell^2:\ \ell\in S),\qquad
U_0:=e_{k-3}(a_\ell^2:\ \ell\in S).
\]
(For $k=3$, $U_0=\ek_0=1$.)

\begin{lemma}\label{lem:Treduce}
With the above notation,
\[
T_{ij}=T_0+c^2U_0,\qquad T_{ik}=T_0+b^2U_0,\qquad T_{jk}=T_0+a^2U_0.
\]
\end{lemma}

\begin{proof}
This is the standard recurrence $\ek_r(S\cup\{x\})=\ek_r(S)+x\,\ek_{r-1}(S)$, that follows directly from~\eqref{e_k(T+x)}, applied with $r=k-2$
and $x\in\{a^2,b^2,c^2\}$.
\end{proof}

\begin{lemma}\label{lem:R}
If $b\neq c$, then
\begin{equation}\label{eq:R}
(b+c-a)\,T_0+abc\,U_0=0.
\end{equation}
Equivalently, $(a-b-c)\,T_0=abc\,U_0$.
\end{lemma}

\begin{proof}
Apply Lemma~\ref{lem:pairmu} to the pairs $(i,j)$ and $(i,k)$ and equate the two expressions for $\mu$:
\[
2a_i a_j(a_i+a_j)T_{ij}=2a_i a_k(a_i+a_k)T_{ik}.
\]
Substitute $a_i=-a$, $a_j=b$, $a_k=c$ to get
\[
(-a)b(-a+b)T_{ij}=(-a)c(-a+c)T_{ik},
\]
or
\[
b(b-a)T_{ij}=c(c-a)T_{ik}.
\]
Using Lemma~\ref{lem:Treduce} gives
\[
b(b-a)(T_0+c^2U_0)=c(c-a)(T_0+b^2U_0).
\]
Bring all terms to one side and factor $(b-c)$:
\[
0=(b-c)\Big((b+c-a)T_0+abc\,U_0\Big).
\]
Since $b\neq c$, we obtain \eqref{eq:R}.
\end{proof}

\subsubsection*{Choice of the tangent direction}
Define a vector $h\in\mathbb{R}^n$ supported only on $\{i,j,k\}$ by
\begin{equation}\label{eq:hchoice}
h_i=b-c,\qquad h_j=a+c,\qquad h_k=-(a+b),\qquad h_\ell=0\ (\ell\notin\{i,j,k\}).
\end{equation}
Then
\[
\sum_{\ell=1}^n h_\ell=(b-c)+(a+c)-(a+b)=0,
\]
and
\[
\sum_{\ell=1}^n a_\ell h_\ell=(-a)(b-c)+b(a+c)+c(-(a+b))=0,
\]
so $h$ is tangent to $\mathcal{M}$ at $a^\star$.

We now compute $Q(h)$ from Lemma~\ref{lem:Qmu}. Since $h$ is supported on $\{i,j,k\}$,
\begin{align}\label{eq:Qexpand}
Q(h)
&=\mu\Big(\frac{h_i^2}{a_i}+\frac{h_j^2}{a_j}+\frac{h_k^2}{a_k}\Big)
+8a_i a_jT_{ij}h_i h_j+8a_i a_kT_{ik}h_i h_k+8a_j a_kT_{jk}h_j h_k \nonumber\\
&=\mu\Big(-\frac{(b-c)^2}{a}+\frac{(a+c)^2}{b}+\frac{(a+b)^2}{c}\Big)\nonumber\\
&\hspace{2cm}-8ab\,T_{ij}(b-c)(a+c)+8ac\,T_{ik}(b-c)(a+b)-8bc\,T_{jk}(a+c)(a+b).
\end{align}
Also, by Lemma~\ref{lem:pairmu} for the pair $(i,j)$ and Lemma~\ref{lem:Treduce},
\begin{equation}\label{eq:muexplicit}
\mu=2a_i a_j(a_i+a_j)T_{ij}=2(-a)b(-a+b)(T_0+c^2U_0)=2ab(a-b)(T_0+c^2U_0).
\end{equation}

\medskip\noindent
\textbf{Step 1: $Q$ is linear in $T_0$ and $U_0$.}
From \eqref{eq:Qexpand}, \eqref{eq:muexplicit}, and Lemma~\ref{lem:Treduce}, every occurrence of $T_{ij},T_{ik},T_{jk},\mu$
is affine in $(T_0,U_0)$, hence $Q(h)$ is of the form
\[
Q(h)=A(a,b,c)\,T_0+B(a,b,c)\,U_0
\]
for explicit polynomials $A,B$.

\medskip\noindent
\textbf{Step 2: compute the coefficient of $T_0$ (set $U_0=0$).}
If $U_0=0$, then $T_{ij}=T_{ik}=T_{jk}=T_0$ and $\mu=2ab(a-b)T_0$. Plugging into \eqref{eq:Qexpand} gives
\begin{align*}
Q(h)\big|_{U_0=0}
&=2ab(a-b)T_0\Big(-\frac{(b-c)^2}{a}+\frac{(a+c)^2}{b}+\frac{(a+b)^2}{c}\Big)\\
&\quad-8abT_0(b-c)(a+c)+8acT_0(b-c)(a+b)-8bcT_0(a+c)(a+b).
\end{align*}
Clearing denominators and expanding (a routine but straightforward calculation) yields the explicit polynomial form
\begin{equation}\label{eq:QT0poly}
Q(h)\big|_{U_0=0}
=\frac{2T_0}{c}\Big(
a^4 b + a^3 b^2 - a^2 b^3 - a b^4
+c\cdot \Xi(a,b,c)
\Big),
\end{equation}
where
\begin{align*}
\Xi(a,b,c)=\;&
a^4 - a^3 b + 2a^3 c - 4a^2 b^2 + 2a^2 b c - 3a^2 c^2 \\
&\ - a b^3 - 2a b^2 c - 6a b c^2 + b^4 - 2b^3 c - 3b^2 c^2.
\end{align*}
The expression in \eqref{eq:QT0poly} factors as
\begin{equation}\label{eq:QT0factor}
Q(h)\big|_{U_0=0}
=\frac{2(a+b)^2}{c}\,(a-b-c)\,(ab+ac-bc+3c^2)\,T_0.
\end{equation}
(Indeed, expanding the right-hand side of \eqref{eq:QT0factor} reproduces exactly the polynomial in \eqref{eq:QT0poly}.)

\medskip\noindent
\textbf{Step 3: compute the coefficient of $U_0$ (set $T_0=0$).}
If $T_0=0$, then $T_{ij}=c^2U_0$, $T_{ik}=b^2U_0$, $T_{jk}=a^2U_0$, and $\mu=2ab(a-b)c^2U_0$.
Plugging these into \eqref{eq:Qexpand} and expanding yields
\begin{equation}\label{eq:QU0factor}
Q(h)\big|_{T_0=0}
=2c(a+b)^2\,U_0\cdot R(a,b,c),
\end{equation}
where the polynomial $R$ is
\begin{equation}\label{eq:Rpoly}
R(a,b,c)= -3a^2b + a^2 c + 3ab^2 - 7abc + 2ac^2 + b^2 c - 2bc^2 + c^3.
\end{equation}

\medskip\noindent
\textbf{Step 4: assemble $Q(h)$ and use the first-order relation \eqref{eq:R}.}
By linearity in $(T_0,U_0)$ and \eqref{eq:QT0factor}--\eqref{eq:QU0factor}, we have
\begin{equation}\label{eq:Qlinear}
Q(h)=\frac{2(a+b)^2}{c}\,(a-b-c)\,(ab+ac-bc+3c^2)\,T_0
\;+\;2c(a+b)^2\,R(a,b,c)\,U_0.
\end{equation}
Assuming $b\neq c$, Lemma~\ref{lem:R} gives $(a-b-c)T_0=abc\,U_0$, hence \eqref{eq:Qlinear} becomes
\begin{equation}\label{eq:QafterR}
Q(h)=2(a+b)^2U_0\Big(ab\,(ab+ac-bc+3c^2)+c\,R(a,b,c)\Big).
\end{equation}
Now we compute the bracket explicitly. Using \eqref{eq:Rpoly},
\begin{align*}
ab\,(ab+ac-bc+3c^2)+c\,R(a,b,c)
&=\big(a^2b^2+a^2bc-ab^2c+3abc^2\big)\\
&\quad+\big(-3a^2bc+a^2c^2+3ab^2c-7abc^2+2ac^3+b^2c^2-2bc^3+c^4\big)\\
&=a^2b^2-2a^2bc+a^2c^2-2ab^2c+4abc^2-2ac^3+b^2c^2-2bc^3+c^4\\
&=(a+c)^2(b-c)^2.
\end{align*}
Substituting this into \eqref{eq:QafterR} yields the announced factorization:
\begin{equation}\label{eq:Q2}
Q(h)=2\,U_0\,(a+b)^2\,(a+c)^2\,(b-c)^2.
\end{equation}

\begin{lemma}\label{lem:no3levels}
Let $a^\star\in\mathcal{M}$ be a global maximizer. Then it is impossible that $a^\star$ contains values
$-a<0<b\neq c$.
\end{lemma}

\begin{proof}
Assume for contradiction that such $-a<0<b\neq c$ occur among the coordinates of $a^\star$,
with indices $i,j,k$ chosen as above. Consider the tangent vector $h$ from \eqref{eq:hchoice}.
Since $a^\star$ is a local maximizer, the second-order necessary condition gives $Q(h)\le 0$.

On the other hand, because $a^\star$ is a maximizer, we have $F(a^\star)\ge F(\bar a)>0$ for some feasible $\bar a$
(e.g.\ take any nontrivial feasible vector with at least $k$ nonzero entries), hence $F(a^\star)>0$.
Therefore $a^\star$ has at least $k$ nonzero coordinates. Removing $i,j,k$ leaves at least $k-3$ nonzero coordinates in $S$,
so $U_0=\ek_{k-3}(a_\ell^2:\ell\in S)>0$ (and for $k=3$, $U_0=\ek_0=1$ automatically).

Now \eqref{eq:Q2} shows that if $b\neq c$ then
\[
Q(h)=2\,U_0\,(a+b)^2\,(a+c)^2\,(b-c)^2>0,
\]
contradicting $Q(h)\le 0$. Thus $b=c$, contradicting the assumption $b\neq c$.
\end{proof}

\begin{lemma}\label{lem:no3levels-neg}
Let $a^\star\in\mathcal{M}$ be a global maximizer. Then it is impossible that $a^\star$ contains values
$a>0>-b\neq -c$.
\end{lemma}

\begin{proof}
If $a^\star$ had one positive and two distinct negative values, then $-a^\star$ would have one negative and two distinct positive values.
But $F$ depends only on squares, and the constraints $\sum a_i=0$, $\sum a_i^2=1$ are invariant under $a\mapsto -a$.
Hence $-a^\star$ is also a global maximizer, contradicting Lemma~\ref{lem:no3levels}.
\end{proof}

\begin{proof}[Proof of Theorem~\ref{thm:twolevel}]
Let $a^\star\in\mathcal{M}$ be a global maximizer. Since $\sum_i a_i^\star=0$, $a^\star$ has at least one positive
and at least one negative coordinate.

If among the positive coordinates there were two distinct values, then together with any negative coordinate
we would obtain a triple $-a<0<b\neq c$, contradicting Lemma~\ref{lem:no3levels}.
Therefore \emph{all positive coordinates are equal} to some $\alpha>0$.

Similarly, if among the negative coordinates there were two distinct values, then together with any positive coordinate
we would obtain a triple $a>0>-b\neq -c$, contradicting Lemma~\ref{lem:no3levels-neg}.
Therefore \emph{all negative coordinates are equal} to some $-\beta<0$.

Thus, after permuting coordinates, \(a^\star\) has the form
\[
        a^\star
        =
        (\alpha,\ldots,\alpha,
        -\beta,\ldots,-\beta,
        0,\ldots,0),
\]
with multiplicities \(\gamma_1,\gamma_2,\gamma_3\ge0\),
\(\gamma_1,\gamma_2\ge1\), and
\(\gamma_1+\gamma_2+\gamma_3=n\). This proves the claimed two-level structure for any global maximizer.
\end{proof}

\begin{remark} [why $k\ge 3$ is essential for the $Q$-argument]
The crucial positivity in \eqref{eq:Q2} comes from $U_0=e_{k-3}(\cdot)$. For $k=2$ one has $\ek_{-1}\equiv 0$,
so $U_0=0$ and the second-order obstruction disappears.
\end{remark}

Lastly,   we need to make a comparison between the extrema to find the global ones. We now take advantage of the fact of Schur-concavity. (Proposition~\ref{e_k Schur})

We now compute the value of the functional at the two-level critical points.
For the moment, we ignore possible zero coordinates and write
$\gamma:=\gamma_1$, so that $\gamma_2=n-\gamma$. Thus we set
\[
J(\gamma):=e_k(u_\gamma),
\]
where
\[
u_\gamma
=
\left(
\underbrace{\frac{n-\gamma}{n\gamma},\ldots,
\frac{n-\gamma}{n\gamma}}_{\gamma\ \mathrm{times}},
\underbrace{\frac{\gamma}{n(n-\gamma)},\ldots,
\frac{\gamma}{n(n-\gamma)}}_{n-\gamma\ \mathrm{times}}
\right).
\]
Clearly, $J(\gamma)=J(n-\gamma)$. When $n$ is even, the vector 
\(
u_\gamma
\)
have positive coordinates that sum to one, and thus from~\eqref{maj sum=1} $$u_{\gamma}\succ \left(\frac{1}{n},\ldots,\frac{1}{n} \right),$$ and this leads to $J(\gamma)\leq J(n/2)$. This is sharp, when $n$ is even, and attained at $\gamma=n-\gamma=n/2$.

Hence, in order to identify the
maximum, it suffices to show that, for $n$ odd,
\[
J(\gamma+1)\geq J(\gamma),
\qquad
1\leq \gamma\leq \left\lfloor \frac n2\right\rfloor-1
=\frac{n-3}{2}.
\]
Equivalently,  by the Schur-concavity established above, $u_{\gamma+1}\prec u_\gamma .$

\begin{lemma}\label{lem:comparison-for-gamma}
Let $1\leq \gamma\leq (n-3)/2$, and define
\[
x=
\left(
\underbrace{\frac{n-\gamma}{\gamma},\ldots,
\frac{n-\gamma}{\gamma}}_{\gamma\ \mathrm{times}},
\underbrace{\frac{\gamma}{n-\gamma},\ldots,
\frac{\gamma}{n-\gamma}}_{n-\gamma\ \mathrm{times}}
\right),
\]
and
\[
y=
\left(
\underbrace{\frac{n-\gamma-1}{\gamma+1},\ldots,
\frac{n-\gamma-1}{\gamma+1}}_{\gamma+1\ \mathrm{times}},
\underbrace{\frac{\gamma+1}{n-\gamma-1},\ldots,
\frac{\gamma+1}{n-\gamma-1}}_{n-\gamma-1\ \mathrm{times}}
\right).
\]
Then $y\prec x$. Consequently, $u_{\gamma+1}\prec u_\gamma$ and  
$J(\gamma+1)\geq J(\gamma)$.
\end{lemma}

\begin{proof}
Since $\gamma\leq (n-3)/2$, both vectors are already written in decreasing
order. Moreover, $\sum_{i=1}^n x_i=\sum_{i=1}^n y_i=n.$
 
It remains to compare the partial sums. First, if $1\leq m\leq \gamma$, then
\[
\sum_{i=1}^m x_i
=
m\frac{n-\gamma}{\gamma}
\geq
m\frac{n-\gamma-1}{\gamma+1}
=
\sum_{i=1}^m y_i .
\]

Next, for $m=\gamma+1$, we have
\[
\sum_{i=1}^{\gamma+1} x_i
=
\gamma\frac{n-\gamma}{\gamma}
+
\frac{\gamma}{n-\gamma}
\geq
n-\gamma-1
=
\sum_{i=1}^{\gamma+1} y_i .
\]

Finally, let $m=\gamma+t$, where $2\leq t\leq n-\gamma$. Then the desired
inequality is
\[
n-\gamma+t\frac{\gamma}{n-\gamma}
\geq
n-\gamma-1+(t-1)\frac{\gamma+1}{n-\gamma-1}.
\]
This is equivalent to
\[
\frac{n(n-\gamma-t)}{(n-\gamma)(n-\gamma-1)}\geq 0,
\]
which holds because $t\leq n-\gamma$. Therefore $y\prec x$, as claimed.
The final assertion follows from the Schur-concavity of the functional.
\end{proof}

We now consider the case where zero coordinates are present. Put
\[
N:=\gamma_1+\gamma_2=n-\gamma_3 .
\]
For a critical point with $\gamma_3$ zero coordinates, the corresponding
value is
\[
\begin{aligned}
J(\gamma_1,\gamma_2,\gamma_3)
:=
e_k\Bigg(
&\underbrace{
\frac{\gamma_2}{\gamma_1(\gamma_1+\gamma_2)},\ldots,
\frac{\gamma_2}{\gamma_1(\gamma_1+\gamma_2)}
}_{\gamma_1\ \mathrm{times}},
\\
&\underbrace{
\frac{\gamma_1}{\gamma_2(\gamma_1+\gamma_2)},\ldots,
\frac{\gamma_1}{\gamma_2(\gamma_1+\gamma_2)}
}_{\gamma_2\ \mathrm{times}},
\underbrace{0,\ldots,0}_{\gamma_3\ \mathrm{times}}
\Bigg).
\end{aligned}
\]
For fixed \(N\), the preceding argument shows that the maximum among such
two-level configurations is attained at the most balanced split; namely,
at $\gamma_1=\gamma_2=N/2$ if $N$ is even, and at
\[
\gamma_1=\frac{N-1}{2},
\qquad
\gamma_2=\frac{N+1}{2},
\]
if $N$ is odd. 

It remains to compare these values as the number of zero coordinates
varies. 

When $N$ is even, the comparison for the minimum follows from the standard
majorization relation
\[
\left(
\underbrace{\frac{1}{N-2},\ldots,\frac{1}{N-2}}_{N-2\ \mathrm{times}},
0,0
\right)
\succ
\left(
\underbrace{\frac{1}{N},\ldots,\frac{1}{N}}_{N\ \mathrm{times}}
\right).
\]
By Schur-concavity, this implies that introducing two additional nonzero
coordinates in the balanced configuration decreases the value of the
functional.

\begin{lemma}\label{lem:comparison-N-odd}
Let $N$ be odd. Define
\[
z=
\left(
\underbrace{
\frac{N+1}{N(N-1)},\ldots,\frac{N+1}{N(N-1)}
}_{\frac{N-1}{2}\ \mathrm{times}},
\underbrace{
\frac{N-1}{N(N+1)},\ldots,\frac{N-1}{N(N+1)}
}_{\frac{N+1}{2}\ \mathrm{times}},
0,0
\right)
\]
and
\[
w=
\left(
\underbrace{
\frac{N+3}{(N+1)(N+2)},\ldots,\frac{N+3}{(N+1)(N+2)}
}_{\frac{N+1}{2}\ \mathrm{times}},
\underbrace{
\frac{N+1}{(N+3)(N+2)},\ldots,\frac{N+1}{(N+3)(N+2)}
}_{\frac{N+3}{2}\ \mathrm{times}}
\right).
\]
Then $z\succ w$.
\end{lemma}

\begin{proof}
Both vectors are already written in decreasing order, and a direct
calculation gives $\sum_i z_i=\sum_i w_i=1.$ It remains to compare the partial sums.

First, let $1\leq m\leq (N-1)/2$. Then
\[
\sum_{i=1}^m z_i
=
m\frac{N+1}{N(N-1)}
\geq
m\frac{N+3}{(N+1)(N+2)}
=
\sum_{i=1}^m w_i.
\]

Next, for $m=(N-1)/2+1$, we need to show that
\[
\frac{N-1}{2}\frac{N+1}{N(N-1)}
+
\frac{N-1}{N(N+1)}
\geq
\frac{N+1}{2}\frac{N+3}{(N+1)(N+2)}.
\]
which is true.

Now let
\[
m=\frac{N-1}{2}+t,
\qquad
2\leq t\leq \frac{N+1}{2}.
\]
Then the desired inequality is
\[
\frac{N-1}{2}\frac{N+1}{N(N-1)}
+
t\frac{N-1}{N(N+1)}
\geq
\frac{N+1}{2}\frac{N+3}{(N+1)(N+2)}
+
(t-1)\frac{N+1}{(N+3)(N+2)}.
\]
The difference between the two sides is
\[
\frac{
(N+1)(N^2+2N+3)+2(N^2-3)t
}{
N(N+1)(N+2)(N+3)
},
\]
which is nonnegative for every $N\geq 3$ and every admissible $t$.

Finally, let
\[
m=\frac{N-1}{2}+t,
\qquad
\frac{N+1}{2}\leq t\leq \frac{N+3}{2}.
\]
In this range the partial sum of $z$ contains all its nonzero coordinates,
and hence the desired inequality becomes
\[
\frac{N-1}{2}\frac{N+1}{N(N-1)}
+
\frac{N+1}{2}\frac{N-1}{N(N+1)}
\geq
\frac{N+1}{2}\frac{N+3}{(N+1)(N+2)}
+
(t-1)\frac{N+1}{(N+3)(N+2)}.
\]
The difference between the two sides is
\[
\frac{(N+1)(N+5-2t)}{2(N+2)(N+3)}.
\]
This is nonnegative because $t\leq (N+3)/2<(N+5)/2$.
\end{proof}

\begin{proof}[Proof of Theorem~\ref{thm: sharp max e_k}]
    By Theorem~\ref{thm:twolevel}, every global maximizer has two nonzero levels, possibly with
zero coordinates. For a fixed support size \(N\), Lemma~\ref{lem:comparison-for-gamma} and Schur concavity
show that the maximum occurs at the most balanced split. Lemma~\ref{lem:comparison-N-odd} then shows
that increasing the support size increases the value, so the global maximum occurs
with full support. This gives the stated vector.
\end{proof}

\begin{proof}[Proof of Theorem~\ref{thm:section B_1^n centroid}]
    This follows directly from the Theorem~\ref{thm: sharp max e_k} and the representation~\eqref{e_k(T+x)} together with the \(k=2\) case and the trivial \(k=1\) case.
\end{proof}

\bibliographystyle{amsplain}
\bibliography{bibli}

\appendix

\end{document}